\documentclass[final, 11pt]{article}

\def\showauthornotes{0}
\def\showkeys{0}
\def\showdraftbox{1}

\usepackage{bm}
\usepackage{bbm}

\usepackage{amsthm}

\usepackage{xspace,xcolor,enumerate}
\usepackage{amsmath,amsfonts,amssymb}
\usepackage{color,graphicx}

\ifnum\showkeys=1
\usepackage[color]{showkeys}
\fi

\definecolor{darkred}{rgb}{0.5,0,0}
\definecolor{darkgreen}{rgb}{0,0.5,0}
\definecolor{darkblue}{rgb}{0,0,0.5}

\usepackage[pdfstartview=FitH,pdfpagemode=None,colorlinks,linkcolor=darkred,filecolor=blue,citecolor=darkred,urlcolor=darkred,pagebackref]{hyperref}

\usepackage{dsfont}

\usepackage[capitalise,nameinlink]{cleveref}

\ifnum\showauthornotes=1
\newcommand{\Authornote}[2]{{\sf\small\color{red}{[#1: #2]}}}
\newcommand{\Authorcomment}[2]{{\sf \small\color{gray}{[#1: #2]}}}
\newcommand{\Authorfnote}[2]{\footnote{\color{red}{#1: #2}}}
\else
\newcommand{\Authornote}[2]{}
\newcommand{\Authorcomment}[2]{}
\newcommand{\Authorfnote}[2]{}
\fi

\ifnum\showdraftbox=1

\else

\fi

\newtheorem{theorem}{Theorem}[section]

\crefname{conjecture}{Conjecture}{Conjectures}
\newtheorem{definition}[theorem]{Definition}
\crefname{definition}{Definition}{Definitions}
\newtheorem{lemma}[theorem]{Lemma}
\crefname{lemma}{Lemma}{Lemmas}
\newtheorem{remark}[theorem]{Remark}
\crefname{remark}{Remark}{Remarks}
\newtheorem{proposition}[theorem]{Proposition}
\crefname{proposition}{Proposition}{Propositions}
\newtheorem{corollary}[theorem]{Corollary}
\crefname{corollary}{Corollary}{Corollaries}

\crefname{obs}{Observation}{Observations}
\newtheorem{claim}[theorem]{Claim}
\crefname{claim}{Claim}{Claims}

\crefname{fact}{Fact}{Facts}

\crefname{openprob}{Open Problem}{Open Problems}

\crefname{remk}{Remark}{Remarks}

\newtheorem{example}[theorem]{Example}
\crefname{example}{Example}{Examples}

\def\to{\rightarrow}
\def\eps{\varepsilon}
\def\epsilon{\varepsilon}

\def\phi{\varphi}

\def\implies{\Rightarrow}

\renewcommand{\bar}{\overline} 

\newcommand{\Q}{{\mathbb Q}}
\newcommand{\R}{{\mathbb R}}
\newcommand{\E}{{\mathbb E}}
\newcommand{\C}{{\mathbb C}}
\newcommand{\N}{{\mathbb{N}}}

\newcommand{\F}{{\mathbb F}}

\usepackage{nicefrac}

\newcommand{\abs}[1]{\ensuremath{\left\lvert #1 \right\rvert}}

\newcommand{\Esymb}{\mathbb{E}}
\newcommand{\Psymb}{\mathrm{Pr}}

\DeclareMathOperator*{\ExpOp}{\Esymb}

\DeclareMathOperator*{\ProbOp}{\Psymb}
\renewcommand{\Pr}{\ProbOp}

\renewcommand{\E}{\ExpOp}

\newfont{\inhead}{eufm10 scaled\magstep1}

\newcommand{\polylog}{{\mathrm{polylog}}}

\DeclareMathOperator\supp{Supp}

\allowdisplaybreaks[1]

\usepackage{color,graphicx}
\usepackage{comment}
\usepackage[linesnumbered,boxed]{algorithm2e}
\usepackage{fullpage,amssymb,amsmath,amsthm}

\ifnum\showkeys=1
\usepackage[color]{showkeys}
\fi

\definecolor{darkred}{rgb}{0.5,0,0}
\definecolor{darkgreen}{rgb}{0,0.5,0}

\renewcommand{\eps}{\varepsilon}
\renewcommand{\epsilon}{\eps}
\newcommand{\set}[2][]{{\left\{#2\right\}}^{#1}}

\newtheorem{THEOREM}{Theorem}

\newcommand{\mb}[1]{\mathbb{#1}}
\newcommand{\mc}[1]{\mathcal{#1}}

\newcommand{\poly}{\mathrm{poly}}

\newcommand{\size}[1]{\abs{#1}}
\newcommand{\mon}[1]{\|{#1}\|}
\newcommand{\condset}[2]{\set{#1 \; \left| \;#2 \right. }}
\newcommand{\restrict}[1]{| _{#1}}
\newcommand{\divs}{\; | \;}
\newcommand{\ndivs}{\not | \;}
\newcommand{\Fn}{\mathbb{F}[x_{1},x_{2},\ldots ,x_{n}]}
\newcommand{\Fna}{\mathbb{F}[x_{1},x_{2},\ldots ,x_{n+1}]}
\newcommand{\Fny}{\mathbb{F}[y,x_{1},x_{2},\ldots ,x_{n}]}
\newcommand{\BigO}{\mathcal{O}}
\newcommand{\eqdef}{\stackrel{\Delta}{=}}
\newcommand{\nequiv}{\not \equiv}

\newcommand{\var}{\mathrm{var}}

\newcommand{\SP}{\mathcal{SP}}
\newcommand{\SB}{\xi}
\newcommand{\cost}{\mathrm{c}}

\newcommand{\cF}{\overline{\F}}
\newcommand{\lc}{\mathrm{lc}}

\newcommand{\RFn} {\F(x_{1},x_{2},\ldots ,x_{n})}
\newcommand{\RFnp}{\F(x^p_{1},x^p_{2},\ldots ,x^p_{n})}
\newcommand{\Const}{\mathrm{C}}

\newcommand{\fy} {f(y,x_{1},x_{2},\ldots ,x_{n})}
\newcommand{\Fpar}[2]{\frac{\partial {#1}}{\partial {x_#2}}}
\newcommand{\parf}{\frac{\partial f}{\partial y}}
\newcommand{\Disc}{\Delta}
\newcommand{\Res}{\mathrm{Res}}
\newcommand{\xb}{\bar{x}}

\newcommand{\ab}{\bar{a}}
\newcommand{\bb}{\bar{b}}

\newcommand{\ignore}[1]{}
\newcommand*{\rom}[1]{\expandafter\@slowromancap\romannumeral #1@}
\newcommand{\convhull}{{CS}}

\title{
 Deterministic Factorization of Sparse Polynomials with Bounded Individual Degree  
}

\author{ Vishwas Bhargava \thanks{Department of Computer Science, Rutgers University, Piscataway, NJ 08854. Email: {\tt vishwas1384@gmail.com}.}
\and
Shubhangi Saraf \thanks{
Department of Mathematics \& Department of Computer Science, Rutgers University, Piscataway, NJ 08854. Research supported in part by NSF grants
CCF-1350572 and CCF-1540634. Email: {\tt shubhangi.saraf@gmail.com}.}
\and 
Ilya Volkovich \thanks{Department of EECS, CSE Division, University of Michigan, Ann Arbor, MI 48109. Email: {\tt ilyavol@umich.edu}. 
.} }

\date{}

\begin{document}

\sloppy

\maketitle

\setcounter{page}{1}
\thispagestyle{empty}
\begin{abstract}
In this paper we study the problem of deterministic factorization of sparse polynomials. We show that if $f \in \Fn$ is a polynomial with $s$ monomials, with individual degrees of 
its variables bounded by $d$, then $f$ can be deterministically factored in time $s^{\poly(d)\log n}$. Prior to our work, the only efficient factoring algorithms known for this class of polynomials were randomized, and other than for the cases of $d=1$ and $d=2$, only exponential time deterministic factoring algorithms were known. 

A crucial ingredient in our proof is a quasi-polynomial sparsity bound for factors of sparse polynomials of bounded individual degree. In particular we show if $f$ is an $s$-sparse polynomial in $n$ variables, with individual degrees of 
its variables bounded by $d$, then the sparsity of each factor of $f$ is bounded by $s^{\BigO({d^2\log{n}})}$. This is the first nontrivial bound on factor sparsity  for  $d>2$.  Our sparsity bound uses techniques from convex geometry, such as the theory of Newton polytopes and  an approximate version of the classical Carath\'eodory's Theorem. 

Our work addresses and partially answers a question of von zur Gathen and Kaltofen (JCSS 1985) who asked whether a quasi-polynomial bound holds for the sparsity of  factors of sparse polynomials.

\end{abstract}



\section{Introduction} 
Polynomial factorization is  one of the most fundamental questions in computational algebra. The problem of multivariate polynomial factorization asks the following: Given $f \in \F[x_1,x_2 \ldots, x_n]$  a multivariate polynomial over a field $\F$, compute each of the irreducible factors of $f$. Other than being natural and central,
the problem has many applications in areas such as list decoding \cite{Sudan97, GS99},  derandomization \cite{KabanetsImpagliazzo04} and cryptography \cite{ChorRiverst88}.

There has been a large body of research studying efficient algorithms for this problem (see e.g. \cite{GathenGerhard99})
and numerous \emph{randomized} algorithms were designed  \cite{GathenKaltofen85, Kaltofen87, Kaltofen89, KaltofenTrager90, GathenGerhard99,Kaltofen03,Gathen06}.
However, the question of whether there exist \emph{deterministic} algorithms for this problem remains an
important and interesting open question (see \cite{GathenGerhard99,Kayal07}). 

Another fundamental question in algebraic complexity is the problem of Polynomial Identity Testing (PIT). The problem of PIT asks the following: Given a polynomial $f \in \Fn$ represented by a small arithmetic circuit, determine if the polynomial is identically $0$. 
In a recent work, Kopparty et al~\cite{KSS14} showed that the problem of derandomizing multivariate polynomial factorization is {\it equivalent} to the problem of derandomizing polynomial identity testing for general arithmetic circuits. They showed this result in both the white-box and the black-box settings. 
We already know deterministic PIT algorithms for several interesting classes of arithmetic circuits, and this raises the very natural question of whether we can derandomize {\it polynomial factoring} for these classes. Perhaps the most natural such class of polynomials is the class of {\it sparse} polynomials. 

The sparsity of $f$, denoted
$\mon{f}$ , is the number of monomials (with non zero coefficients) appearing in $f$. For instance,
the sparsity of the polynomial $x_1 +x_2^3+x_3x_4+20$ is four. 

Factoring of sparse polynomials has been studied for over three decades. It was initiated by the work of von zur Gathen and Kaltofen~\cite{GathenKaltofen85} that gives the first {\it randomized} algorithm for factorization of sparse multivariate polynomials. The runtime of this algorithm has polynomial dependence on the sparsity of the {\it factors} of the underlying polynomial, and thus, very naturally, this work raised the question of whether one can find efficient bounds on the sparsity of factors of a sparse polynomial.


In this paper, we consider the following two problems:
(1) Prove efficient bounds on the sparsity of the factors of sparse polynomials. (2) Derandomize polynomial factorization for sparse polynomials\footnote{These questions were raised as important open questions in a recent survey by Forbes and Shpilka~\cite{FS15}.}.  


Indeed, these are extremely natural questions to study. However already for general fields, we know that one cannot hope to prove a strong sparsity bound for the factors of a sparse polynomial. (We discuss two interesting examples of polynomials whose factors have a big blow-up in the number of monomials in Section~\ref{sec:tightness}). 

In this paper, we focus our attention on the class of sparse polynomials with bounded individual degree, i.e. for some parameter $d$, we limit the degree of each variable $x_i$ to be at most $d$. 

One very interesting such class of polynomials is the class of sparse multilinear polynomials ($d=1$). This is the simplest case of sparse polynomials with bounded degree. In \cite{ShpilkaVolkovich10}, Shpilka and Volkovich gave a derandomization for the problem of polynomial factorization for this class. Factor sparsity bounds are fairly easy to show
for this class of polynomials, and armed with the sparsity bound and a technique for derandomizing a certain PIT problem that arises, they were able to derandomize factoring in this case. This was extended to the case $d=2$ in the work of Volkovich \cite{Volkovich17}, again by first showing a sparsity bound for the factors of polynomials of individual degree $2$, and then showing how to derandomize the polynomial factorization problem. 
For $d>2$, the techniques used by the above works for proving sparsity bounds on the factors of a polynomial seem 
to break down. 

In a recent beautiful work, Oliveira ~\cite{Oliveira15} showed that the factors of sparse polynomials of bounded individual degree  can be computed by small {\it depth-7} circuits. This again raises the very natural question: What is the size of the best {\it depth-2} circuit computing the factors of a sparse polynomial of bounded individual degree. This is precisely the problem of proving sparsity bounds for the factors of a sparse polynomial of bounded individual degree, which is a question we study in this paper. 

The other question that we address in this work is the problem of deterministically factoring sparse polynomials of bounded individual degree. A bound on the sparsity of the factors of such a polynomial just implies that the factors will have an efficient representation as a sum of monomials. However in order to actually obtain the factors deterministically, there are several additional derandomization hurdles to overcome.

\subsection{Our Results}
In this paper we give the first deterministic quasi-polynomial time algorithm for factoring sparse polynomials of bounded individual degree. Prior to our work, only efficient randomized factoring algorithms were known for this class of polynomials, and other than for the cases of $d=1$ \cite{ShpilkaVolkovich10} and $d=2$ \cite{Volkovich17} only exponential time deterministic factoring algorithms were known. 

A crucial ingredient of our proof is a quasi-polynomial size sparsity bound for factors of sparse polynomials of bounded individual degree $d$.
In particular, we show that if $f$ is an $s$-sparse polynomial in $n$ variables with individual degrees of 
its variables bounded by $d$, then $f$ can be deterministic factored in time $s^{\poly(d)\log n}$. This is the first nontrivial bound on factor sparsity for any $d>2$.  Our sparsity bound uses techniques from convex geometry, such as the theory of Newton polytopes and  an approximate version of the classical Carath\'eodory's Theorem.

We say that a polynomial $f \in \Fn$ has \emph{sparsity} $s$ if it has at most $s$ nonzero monomials. We say that it has individual degree at most $d$ if the maximum degree in each of its variables is bounded above by $d$.

We formally state below our factor sparsity bound and then our result on deterministic factoring. 

\begin{THEOREM}[Factor Sparsity Bound]
\label{THM:Main1}
Let $\F$ be an arbitrary field (finite or otherwise) and let $f \in \Fn$ be a polynomial of sparsity $s$ and individual degrees at most $d$, then the sparsity of every factor of $f$ is bounded by $s^{\BigO({d^2\log{n}})}$.
\end{THEOREM}

\begin{remark}
Note that for $d = \polylog (n)$, we obtain a quasi-polynomial sparsity bound on the factors of $f$. Indeed when $s = \poly(n)$, for any $d= o(\sqrt n/\log^2 n)$, we obtain a nontrivial sparsity bound on the factors of $f$. 
\end{remark}

Given a polynomial $f \in \Fn$, the \emph{complete factorization} of $f$ is a representation of $f$ as a product $h_1^{e_1} \cdots h_m^{e_m}$, where $h_{1}, h_{2}, \ldots, h_{m}$-s are pairwise coprime, irreducible polynomials, and $e_1, e_2, \ldots, e_m$ are positive integers. This representation is unique up to reordering of the $h_i$. 

\begin{THEOREM}[Main]
\label{THM:Main2}
There exists a deterministic algorithm that given a polynomial $f \in \Fn$ of sparsity $s$ and individual degrees at most $d$, computes the complete factorization of $f$, using 
$s^{\BigO({d^7\log{n}})} \cdot \poly(\cost_{\F}(d^2))$ field operations, where: 
\begin{enumerate}
\item $\cost_{\F}(d) = \poly(\ell \cdot p,d)$, if $\F=\F_{p^\ell}$.
\item $\cost_{\F}(d) = \poly(d, t)$,  where $t$ is maximum bit-complexity of the coefficients of $f$, if $\F=\mathbb{Q}$.
\end{enumerate}

\end{THEOREM}

\begin{remark}

In the statement of Theorem~\ref{THM:Main2}, $\cost_{\F}(d)$ denotes the time of the best known algorithm that factors a univariate polynomial of degree $d$ over $\F$.

\end{remark}
\begin{remark}
A more refined version of Theorem~\ref{THM:Main2} is given in Theorem~\ref{thm:fact sparse}. The run time for the deterministic factoring algorithm in Theorem~\ref{thm:fact sparse} gives the precise dependence on the sparsity bound for factors of sparse polynomials. In particular, if one could improve the sparsity bound, then one could plug it into the statement of Theorem~\ref{thm:fact sparse} to get an improved run time for the deterministic factoring algorithm.
\end{remark}


\subsection{Related Work}

Over the last three decades, the question of derandomizing sparse polynomial factorization has seen only very partial progress. 

The study of sparse polynomial factorization was initiated in~\cite{GathenKaltofen85}, where the first {\it randomized} algorithm for the factorization of sparse polynomials was given. The runtime of this algorithm was polynomial in the sparsity of the factors, and in this work, von zur Gathen and Kaltofen explicitly raised the question of proving improved sparsity bounds for the factors of sparse polynomials.

In~\cite{DvirOliveira14}, Dvir and Oliveira gave an elegant approach for bounding the sparsity of factors of a general sparse polynomial by studying the Newton polytopes of the polynomial and its factors. This approach did not eventually lead to an efficient sparsity bound. However it did inspire our work and our approach of using techniques from convex geometry to bound the factors of sparse polynomials. 

In \cite{ShpilkaVolkovich10}, Shpilka and Volkovich gave efficient deterministic factoring algorithms for sparse multilinear polynomials.
This result was extended in \cite{Volkovich15a} to the model of sparse polynomials that split into multilinear 
factors. In \cite{Volkovich17}, Volkovich gave an efficient deterministic factorization algorithm for sparse multiquadratic polynomials. The results~\cite{ShpilkaVolkovich10, Volkovich17} correspond to the special case when the individual degree $d$ equals $1$ and $2$, respectively. For $d \geq 3$, the proof techniques of both these works broke down, and a new approach was needed. 

The problem of multivariate polynomial factorization for polynomials of bounded individual degree was also studied in~\cite{Oliveira15}. In this work, among other things, it was shown that if $f \in \Fn$ is an $s$-sparse polynomial  of individual degree $d$, where $\F$ is a field of characteristic $0$, then any factor of $f$ can be computed by a depth-$7$ circuit of size $\poly (dn^d, s)$.
In particular if $d$ is constant, then this shows that any factor of $f$ can be computed by a depth-$7$ circuit with only a polynomial blow-up in size. This is in contrast to our work, where we want to bound the number of {\it monomials} in the factors of $f$. In other words, we attempt to represent the factors of $f$ by the more natural class of depth-$2$ circuits and then understand the size complexity (which we show is quasi-polynomial). We also would like to point out that our result holds over any field $\F$. 

Another work that is relevant in this context is the work of Kopparty et al~\cite{KSS14} which shows an equivalence between the problems of polynomial identity testing (PIT) and polynomial factorization. In particular, it shows that if one can derandomize PIT for the class of general arithmetic circuits, then one can derandomize polynomial factorization for that same class. Since there are several natural examples of classes of polynomials for which we know deterministic PIT algorithms, this naturally raises the question (which was indeed raised in~\cite{KSS14}) of whether one can derandomize factoring for the corresponding classes of polynomials. Sparse polynomials are, perhaps, the most natural example of such a class, and our work makes the first significant advance in this direction.


\subsection{Proof overview}
Our proof of the deterministic factoring algorithm has two self- contained and independently interesting components. We first prove a sparsity bound on the factors of sparse polynomials with bounded individual degree (Theorem \ref{THM:Main1}). We then show how such a sparsity bound can be used effectively to derandomize factoring of this same class of polynomials (Theorem \ref{THM:Main2}). 

We elaborate on both these components below. 

\subsubsection{Proof Overview for the Sparsity Bound: Theorem \ref{THM:Main1}}

Our proof uses tools from convex geometry such as the theory of Newton polytopes and an approximate version of Carath\'eodory's theorem. 

Suppose that $f, g, h \in \Fn$ are polynomials such that $f = g\cdot h$. We want to show that if $f$ is $s$-sparse and with bounded individual degree $d$, then $g$ and $h$ are both at most $s'$ sparse, where $s' = s^{\BigO({d^2\log{n}})}$. 

We will show this by instead showing the following slightly more general result. For a polynomial $f$, let $\mon{f}$ denote the sparsity (i.e. the number of nonzero monomials) of $f$. 
Suppose that $g$ is any polynomial of individual degree $d$ such that $\mon{g} = s$, and suppose that $f = g\cdot h$  (with no assumptions on the degrees of $f$ and $h$), then $\mon{f} \geq s^{\frac{1}{\BigO(d^2\log n)}}$. 
In particular, there is no polynomial $h$ that one can multiply $g$ with, so that the product $g\cdot h$ has an overwhelming cancellation of monomials. 

\paragraph{Newton Polytopes and Connection to the Sparsity Bound}
Let $f \in \Fn$ be a polynomial such that:
 $$f = \sum a_{i_1i_2 \ldots i_n}x_1^{i_1}x_2^{i_2}\cdots x_n^{i_n}.$$ 

One can consider the set $$\supp(f) = \condset {(i_1, i_2, \ldots, i_n)}{  a_{i_1i_2 \ldots i_n} \neq 0} \subseteq \R^n$$ of exponent vectors of $f$. One can then associate a polytope $P_f \subseteq \R^n$, called the Newton polytope of $f$, which is the convex hull of points in $\supp(f)$.

A classic fact about Newton polytopes that was observed by Ostrowski~\cite{Ostrowski1921} in 1921 states that if $f = g \cdot h$, then $P_f$ is the Minkowski sum of $P_g$ and $P_h$, where for two polytopes $A$ and $B$, their Minkowski sum $A+B$ is defined to be the set of points $\condset{u+v}{u \in A \mbox{ and } v \in B}$. Minkowski sums of polytopes are extremely well-studied and it is not difficult to show that the Minkowski sum of two polytopes is itself a polytope. Moreover, if we let $V(P)$ denote the set of vertices (equivalently corner points) of a polytope $P$, then $$\size{V(A+B)} \geq \max \set{ \size{V(A)}, \size{V(B)} }.$$

Once we have these basic facts about Newton polytopes and Minkowski sums, it follows that a lower bound for $\mon{f}$ (in terms of $\mon{g}$), follows from a lower bound on $\size{V(P_f)}$, and in particular from a lower bound on $\size{V(P_g)}$.
Thus, via the theory of Newton polytopes and Minkowski sums, we see that the monomials of $g$ that correspond to the vertices of $P_g$ are very {\it robust}. There is no way of multiplying $g$ with any other polynomial and obtaining a cancellation of monomials that will make these special monomials corresponding to the vertices of $P_g$ ``disappear".

Thus for $f = g \cdot h$, our task of lower bounding $\mon{f}$ in terms of $\mon{g}$ has reduced to lower bounding $\size{V(P_g)}$, where $P_g$ is the Newton polytope of a polynomial $g$ such that $\mon{g} = s$ and $g$ has individual degree bounded by $d$. Showing a lower bound on $\size{V(P_g)}$ will be the main technical core of our proof of the sparsity bound. 

We note that this connection between Newton polytopes and sparsity bounds was first made in~\cite{DvirOliveira14} and indeed it inspired the approach taken in this paper.

\paragraph{Easy Example with Multilinear Polynomials}
We demonstrate the approach of using Newton Polytopes for proving sparsity bounds via the following ``toy'' example of showing a sparsity bound for multilinear polynomials (i.e, when individual degree is bounded by 1). 
Suppose that $f, g, h \in \Fn$ and $g$ is a multilinear polynomial such that $f$ is nonzero and $f = g\cdot h$. One can give the following easy proof by induction on $n$, of the fact that $\mon{f} \geq \mon{g}$. Let $x_1$ be  variable that both $g$ and $h$ depend on. (If such a variable doesn't exist then the sparsity bound is trivial.)
Moreover assume WLOG that $x_1$ doesn't divide $h$, because if it did, then we could factor it out and work with the resulting polynomials. 
Since $g$ is multilinear, we can express $g$ as $g = g_1x_1 + g_0$, where $g_1$ and $g_0$ are multilinear polynomials not depending on $x_1$, and $g_1$ is nonzero.  Let $h = h_dx_1^d + \ldots +h_0$, where $h_d$ and $h_0$ are nonzero, and the $h_i$ don't depending on $x_1$. Now,

\begin{equation*}
f =(g_1x_1 + g_0)\cdot (h_dx_1^d + \ldots +h_0) 
 = (g_1\cdot h_d)x_1^{d+1} + \ldots +(h_0 \cdot g_0).
\end{equation*}

Thus, $\mon{f} \geq \mon{g_1\cdot h_d} + \mon{h_0 \cdot g_0}$. By the induction hypothesis, $\mon{g_1\cdot h_d} \geq \mon{g_1}$ and $\mon{h_0 \cdot g_0} \geq \mon{g_0}$. It follows that $\mon{f} \geq \mon{g_1} + \mon{g_0}  = \mon{g}$.

Now let us give an alternate proof of the above bound using Newton polytopes. Let $\supp(g) \subseteq {\set{0,1}}^n $ be the set of exponent vectors of $g$. Then notice that no element of $\supp(g)$ can be written as a nontrivial convex combination of any other points of $\supp(g)$. In particular, {\it every} element of $\supp(g)$ is a vertex of $P_g$! Thus $\mon{f} \geq \size{V(P_f)} \geq \size{V(P_g)} = \size{\supp(g)} = \mon{g}$. 

\paragraph{Sparsity Bound from Carath\'eodory's Theorem}

Note that in general, for an arbitrary polynomial $g$, there is no good bound on the number of vertices of $P_g$ in terms of the number of monomials of $g$. For instance one can easily construct examples of polynomials $g$ with exponential in $n$ many monomials, and such that $P_g$ has only $n$ vertices. Here is an example : consider the polynomial $P_g = (x_1 + x_2 + \cdots + x_n)^n$. It clearly has exponentially many monomials. However $P_g$ has only $n$ vertices, which are the scalings of the coordinate vectors by $n$. 


In the case when $g$ has individual degree bounded by $d$, we will show that a much nicer bound actually holds. 
Notice that in this case, $\supp(g) \subseteq {\set{0,1,\ldots,d}}^n$. 
We will show that if $E \subseteq {\set{0,1,\ldots,d}}^n$ is an arbitrary subset of size $s$, then the convex hull of $E$ (denoted \convhull(E)) has at least $s^{\frac{1}{d^2 \cdot \log n}}$ vertices. 
This will immediately imply our sparsity bound. 

To show this bound, we will use an approximate version of Carath\'eodory's theorem. The classic version of Carath\'eodory's theorem is a fundamental result in convex geometry and it states that if a point $\mu \in \R^n$ lies in the convex hull of a set $V$, then $\mu$ can be written as the convex combination of at most $n+1$ points of $V$. 

Now, for a set $E \subseteq {\set{0,1,\ldots,d}}^n$, let $V(E)$ denote the vertices of the convex hull of $E$. It is easy to see that $V(E) \subseteq E$. Since every point $\mu \in E$ is a convex combination of elements of $V(E)$, by Carath\'eodory's theorem, it is a convex combination of at most $n+1$ elements of $V(E)$. 
Now $E$ is not an arbitrary collection of points. It is a subset of ${\set{0,1,\ldots,d}}^n$. Suppose we could show the following strengthened (and wishful) Carath\'eodory's theorem in this setting: 
For $E \subseteq {\set{0,1,\ldots,d}}^n$, every point $\mu \in E$ is a convex combination of at most $k$ elements of $V(E)$, where $k$ is some bound much smaller than $n+1$. Not only this, the convex combination is a {\it k-uniform} convex combination, i.e. all the coefficients in the convex combination are equal to $1/k$. 
Notice that in such a case, we can immediately conclude that $\size{E} \leq \size{V(E)}^k$, since each subset of $V(E)$ of size $k$ would ``recover" at most one element of $E$ via a k-uniform convex combination, and each element of $E$ must be recovered by some subset of $V(E)$ of size $k$. If such a result were true for $k \leq d^2 \log n$ then it would imply our sparsity bound!

It unfortunately (and not too surprisingly) turns out that such a wishful theorem is not true. (Though one needs to work a little to find a counterexample.)

However, very fortunately, something very close does end up being true, and it suffices for  our purpose! A suitable ``approximate'' version of Carath\'eodory's theorem suitably applied implies the following:
For $E \subseteq {\set{0,1,\ldots,d}}^n$, every point $\mu \in E$ can be {\it $\epsilon$-approximated} by a $k$-uniform convex combination of elements of $V(E)$, where $k = \BigO(d^2 \log n)$. 
Again, (and this time truly) one can conclude that  $\size{E} \leq \size{V(E)}^k$, since each subset of $V(E)$ of size $k$ could ``approximately recover" at most one element of $E$ via a k-uniform convex combination (by the triangle inequality the same point cannot approximate two different points of $E$), and each element of $E$ must be approximately recovered by some subset of $V(E)$ of size $k$. 
See Theorem~\ref{thm:approx-caratheodory} for the statement of the approximate Carath\'eodory theorem that we use.

\subsubsection{Proof Overview for the Factoring Algorithm: Theorem \ref{THM:Main2}}
\label{sec:proof 2}

Let $f \in \Fny$ be a multivariate polynomial with individual degrees at most $d$. While in general $f$ could have as many as $d(n+1)$ factors, our starting point is an observation that if $f$ is monic\footnote{a polynomial is \emph{monic} in a variable $x_i$ if the leading coefficient of highest degree of $x_i$ in $f$ is equal to $1$.See definition \ref{def:deg} for more details.} in $y$, then every factor of $y$ must also be monic in $y$. Consequently, $f$ has at most $d$ factors (total). 
This makes the monic case much easier to handle, and we first show how to factorize $f$ when $f$ is monic, and then we show how to extend our algorithm to the general non-monic case. 

In the monic case, there are at most $d$ factors. How would we identify these factors? The traditional approach  
\cite{GathenKaltofen85,Kaltofen89,KaltofenTrager90} suggests projecting the polynomial into a low-dimensional space, where the factorization problem is easy. Yet, in order to recover the original factors, the factorization ``pattern'' of $f$ should stay the same upon the projection. That is, every irreducible factor should remain irreducible upon the projection. This is typically achieved by the Hilbert Irreducibility Theorem, which shows that a random projection would achieve this goal.  Nonetheless, derandomizing the Irreducibility Theorem appears to be a challenging task. Instead, we take a somewhat different approach. 

\paragraph{Finding a ``good'' Projection}

First, we relax the requirement of maintaining the same factorization ``pattern'' to a requirement that different irreducible factors do not ``overlap'' upon projection (i.e. have no non-trivial gcd). This is a standard processing step in many factorization algorithms and it is usually taken care of by hitting the Discriminant of the polynomial $f$ (i.e. $\Delta_{y}(f)$). Yet this approach for obtaining our deterministic algorithm presents its challenges, and it is particularly tricky in the case that the characteristic of the ambient field $\F$ is finite (i.e. $\mathrm{char}(\F) > 0$). We show how to go around these problems. 

Formally, let $f \in \Fny$  be monic in $y$ and let $f(y,\xb) =  h_1^{e_1}(y,\xb)  \ldots  h_k^{e_k} (y,\xb)$ be the factorization of $f(y,\xb)$. We will project $f$ to a univariate polynomial in $y$ by setting all the variables in $\xb$ to elements of $\F$. In order to guarantee that different irreducible factors have no non-trivial gcd after projection, it suffices to find an assignment $\ab \in \F^n$ such that $$\forall  i \neq j: \gcd \left( h_i(y,\ab) ,  h_j(y,\ab) \right) = 1.$$ 
This condition translates into finding a single assignment $\ab$ that hits (i.e. is a nonzero assignment for) the \emph{Resultant}, $\Res_y (h_i,h_j)$, for all $i \neq j$ (see Section \ref{sec:res} for more details). As $f$ is an $s$-sparse polynomial, by Theorem \ref{THM:Main1}, each $h_i$ is an $s^{\BigO({d^2\log{n}})}$-sparse polynomial. Hence, by the properties of the Resultant (Lemma \ref{lem:res}), $\Res_y (h_i,h_j)$ is $s^{\BigO({d^3\log{n}})}$-sparse polynomial. Consequently, hitting all the pairwise resultants corresponds to hitting their product, which is a (somewhat) sparse polynomial. We handle this in a ``black-box'' fashion. That is, we iterate over all the points in a hitting set for (somewhat) sparse polynomials (for example using the hitting set of \cite{KlivansSpielman01}). 

\paragraph{Finding the ``right'' Partition}
As the projection we obtain is no longer required to maintain the same factorization ``pattern'', irreducible factors could split into ``pieces'' (i.e. further factorize upon projection) in a way that the same set of ``pieces'' can emerge from different polynomials. For example, consider the polynomials $f(y,x) = (y^2-x)y$ and $g = y(y-x)(y+x)$. These two polynomials have different factorization patterns. However observe that $f(y,1) = g(y,1) = y(y-1)(y+1)$.
While in both cases, the different ``pieces''  of the irreducible factors of $f$ and $g$ do not overlap (i.e. no nontrivial gcd), it is not clear how to group the pieces together to recover the factorization pattern of the original polynomial. I.e. just by examining the pieces, we cannot determine what the right partition of the set of factors of $f(y,1)$ and $g(y,1)$ should be.

We address this problem by recalling and taking advantage of the fact that a monic polynomial of degree $d$ can split into at most $d$ pieces! Therefore, the are at most $d^{\BigO(d)}$ possible partitions. We find the ``right'' partition by iterating over all of them till we find the right one.

\paragraph{Reconstructing the Factors}
As before, let $f \in \Fny$ be monic in $y$ and let $f(y,\xb) =  h_1^{e_1}(y,\xb)  \ldots  h_k^{e_k} (y,\xb)$.
Given a ``good'' projection $\ab$ and the ``right'' partition, we will show how to obtain oracle (i.e. ``black-box'') access the polynomials $h_1, \ldots h_k$. Once we can do this, as Theorem \ref{THM:Main1} provides us an upper bound on the sparsity of $h_i$-s, we can use a reconstruction algorithm for sparse polynomials to reconstruct $h_1, \ldots h_k$, given via an oracle access.

We obtain oracle access to $h_1, \ldots h_k$ by mirroring the factorization algorithm of \cite{KaltofenTrager90}. Given an input point $\bb \in \F^n$ at which we want to compute $h_1 (y, \bb), \ldots h_k(y, \bb)$, the algorithm uses $\ab$ as an anchor point and draws a line to $\bb$. We then obtain a problem of bi-variate factorization, which we know how to solve efficiently. The non-overlapping property of the ``pieces'' makes it possible to group the pieces together in the same consistent way for every choice of $\bb$. Once we can do this, this allows us to evaluate the individual factors at $\bb$. 

\paragraph{Testing the Purported Factors}
As was discussed earlier, given a polynomial $f$, the algorithm will proceeding by trying to reconstruct the factors of $f$ for every projection and every partition. Some of these projections and partitions will return valid factorizations of $f$ and some might return garbage. We need to prune out the garbage solutions, which we can do as follows: As each factor of $f$ is ``somewhat'' sparse 
(Theorem \ref{THM:Main1}) and there are at most $d$ of them, given a purported factorization, we can test if it is a good and valid factorization it by explicitly multiplying out the polynomials. 

Clearly, this algorithm will pick up \emph{any} valid factorization of $f$ (not just the irreducible one). We will select the irreducible factorization using the simple characterization given in Lemma \ref{lem:max fact}.

\paragraph{Factoring General Sparse Polynomials}
In order to the extend the above algorithm that works in the monic case to the more general case of non-monic polynomials, we use a standard reduction that transforms a general polynomial $f \in \Fn$ into a monic one $\hat{f}$. 

More formally, write:  $f = \sum \limits _{j=0}^k f_j \cdot x_{n}^j$ such that $f_k \nequiv 0$ and the $f_j$-s do not depend on $x_n$.
Consider the polynomial $\hat{f}(y, x_1, \ldots, x_{n-1}) =  
f_{k}^{k-1} \cdot f(x_1, \ldots ,x_{n-1}, \frac{y}{f_k}).$
We show that if $f$ is an $s$-sparse polynomial with individual degrees at most $d$, then $\hat{f}$ is an $(s^d)$-sparse polynomial, monic in $y$, with individual degrees at most $d^2$.

Finally, we show that $\hat{f}$ contains all the factors of $f$ that depend on $x_n$, while $f_k$ contains the remaining factors. We recover these remaining factors by recursively factoring $f_k$. Observe that $f_k$ depends on at most $n-1$ variables.

\paragraph{Organization of Paper}
In the next section, we recall some algebraic tools and algebraic algorithms that will be useful for us. 
In Section~\ref{sec:Polytope}, we discuss properties of polytopes and their relation to factor sparsity. Section~\ref{sec:sparsity} contains the proof of the sparsity bound along with a discussion on its tightness.  
We present and analyze the deterministic factoring algorithm in Section~\ref{sec:factAlgo}.
We conclude with some open questions in Section~\ref{sec:open}.

\section{Preliminaries} 

\subsection{Algebraic Tool Kit}

\label{sec:prelim}

Let $\F$ denote a field, finite or otherwise, and let $\cF$ denote
its algebraic closure.  

\subsection{Polynomials}

A polynomial $f \in \Fn$ \emph{depends} on a variable
$x_i$ if there are two inputs $\bar\alpha, \bar\beta \in \bar{\F}^n$
differing only in the $i^{th}$ coordinate for which
$f(\bar\alpha) \neq f(\bar\beta)$.  We denote by $\var(f)$ the set of
variables that $f$ depends on. We say that $f$ is $g$ are \emph{similar} and denote by it $f \sim g$ if $f = \alpha g$ for some $\alpha \neq 0 \in \F$.

For a polynomial $f(x_1,\ldots,x_n)$, a variable $x_i$ and
a field element $\alpha$, we denote with $f \restrict{x_i = \alpha}$ the polynomial resulting from substituting $\alpha$ to
$x_i$. Similarly given a subset $I\subseteq [n]$ and an
assignment $\bar{a}$ $\in \F^{n}$, we define $f \restrict{\xb_I = \bar{a}_I}$ to be the polynomial resulting from substituting $a_i$
to $x_i$ for every $i \in I$.

\begin{definition}[Line]\label{def:line}
Given $\ab, \bb \in \F^n$ we define a \emph{line} passing through $\ab$ and $\bb$ as $\ell_{\ab,\bb} : \F \to \F^n$,
$\ell_{\ab,\bb}(t) \eqdef (1-t) \cdot \ab + t \cdot \bb$. In particular, $\ell_{\ab,\bb}(0) = \ab$
and $\ell_{\ab,\bb}(1) = \bb$.
\end{definition}

\begin{definition}[Degrees, Leading Coefficients]
\label{def:deg}
Let $x_i \in \var(f)$. We can write:  $f = \sum _{j=0}^d f_j \cdot x_i^j$ such that $\forall j: x_i \not \in \var(f_j)$
and $f_d \nequiv 0$. The \emph{leading coefficient} of $f$ w.r.t to $x_i$ is defined as $\lc_{x_i}(f) \eqdef f_d$.
The \emph{individual degree} of $x_i$ in $f$ is defined as $\deg_{x_i}(f) \eqdef d$. We say that $f$ is \emph{monic} in a variable $x_i$
if $\lc_{x_i}(f) = 1$. We say that $f$ is monic if it is monic in some variable. 
\end{definition}

It easy to see that for every $f,g \in \Fn$ and $i \in [n]$ it holds: 
$\lc_{x_i}(f \cdot g) = \lc_{x_i}(f) \cdot \lc_{x_i}(g)$.

\subsection{Factors and Divisibility}

Let $f,g \in \Fn$ be polynomials. We say that $g$ \emph{divides} $f$, or 
equivalently $g$ is a factor of $f$, and denote it by $g \divs f$
if there exists a polynomial $h \in \Fn$ such that $f = g \cdot h$.
We say that $f$ is \emph{irreducible} if $f$ is non-constant
and cannot be written as a product of two non-constant polynomials.

Given the notion of divisibility, we define the gcd
of a set of polynomials in the natural way: we define it to be the highest degree polynomial dividing them all (suitably scaled)\footnote{Such a polynomial is unique up to scaling, and one can fix a canonical polynomial in this class for instance by requiring that the leading monomial has coefficient 1. With this definition, two polynomials are pairwise coprime if their gcd is of degree $0$, and in particular the gcd equals $1$.}. 
Given the notion of irreducibility we can state the important property of the uniqueness of factorization.

\begin{lemma}[Uniqueness of Factorization]
\label{lem:ufd}
Let $h_1^{e_1} \cdot \ldots \cdot h_k^{e_k} = g_1^{e'_1} \cdot \ldots \cdot g_{k'}^{e'_{k'}}$ be
two factorizations of the same non-zero polynomial into irreducible, pairwise coprime factors. Then $k=k'$
and there exists a permutation $\sigma:[k] \to [k]$ such that $h_i \sim g_{\sigma(i)}$ 
and $e_i = e'_{\sigma(i)}$ for $i \in [k]$.
\end{lemma}

Suppose that $f$ is monic in $x_i$. It is easy to see $f$ can be written as a product of monic factors. Therefore, we can specialize Lemma \ref{lem:ufd} to consider the \emph{unique monic factorization} of $f$ as: $f =  h_1^{e_1} \cdot \ldots \cdot h_k^{e_k}$
where $h_i$-s are irreducible, monic, pairwise coprime factors. 

The following lemma provides a characterization of all irreducible, pairwise coprime factorizations of any polynomial.

\begin{lemma}
\label{lem:max fact}
Consider the function $\Phi: \N^{*} \to \N$: given $\bar{e} = (e_1, \ldots, e_k)$, $\Phi(\bar{e}) \eqdef 2 \cdot \sum \limits_{i=1}^k e_i - k$. Let $f \in \Fn$ be a polynomial and let $f = h_1^{e_1} \cdot \ldots \cdot h_{k}^{e_{k}}$ a factorization of $f$ (not necessarily irreducible or coprime), where $h_i$-s are non-constant and $e_i \geq 1$. Then all irreducible, pairwise coprime factorizations of $f$ correspond to those that maximize $\Phi(\bar{e})$.
\end{lemma}

\begin{proof}
First, observe that by Uniqueness of Factorization, all the all irreducible, pairwise coprime factorizations of $f$ result in the same value of $\Phi(\bar{e})$. Next, we show that in the factorization that maximize $\Phi(\bar{e})$, all the $h_i$-s must be irreducible and coprime.
Assume the contrary. We have two possible cases:

\begin{itemize}
\item There exists $i$ such that $h_i$ is reducible. That is, $h_i$ can be written as $h_i = u_i \cdot v_i$, where $u_i, v_i$ are non-constant polynomials. Now, consider a different factorization of $f$ where we replace $h_i^{e_i}$ by $u_i^{e_i}$ and $v_i^{e_i}$. The value of $\Phi$ under the new factorization will increase by $2e_i - 1 \geq 1$. 

\item There exists $i$ and $j$ such that $h_i$ are $h_j$ are not coprime. We can assume w.l.o.g that both $h_i$ and $h_j$ are irreducible. Therefore, $h_i = \alpha \cdot h_j$ for some $\alpha \in \F.$ Consider a different factorization of $f$ where we replace $h_i^{e_i}$ and $h_j^{e_j}$ by a single factor: $(\alpha^{\frac{e_i}{e_i+e_j}} \cdot h_j)^{e_i+e_j}$. The value of $\Phi$ under the new factorization will increase by $1$.
\end{itemize}
\end{proof}


\subsection{Sparse Polynomials}
\label{sec:sparse}

In this section we discuss sparse polynomials, their properties and some related efficient algorithms which leverage 
these properties.

An \emph{$s$-sparse polynomial} is polynomial with at most $s$ (non-zero) monomials. 
We denote by $\mon{f}$ the \emph{sparsity} of $f$.
In this section we list several results related to sparse polynomials.
We begin with an efficient reconstruction algorithm for sparse polynomials.




\begin{lemma}[\cite{KlivansSpielman01}]
\label{lem:rec}
Let $n,s,d \in \N$. There exists a deterministic algorithm that given $n,s,d$ and an oracle access to an $s$-sparse polynomial $f \in \Fn$ of degree $d$, uses $\poly(n,s,d,\log \size{\F})$ field operations and outputs $f$ (in its monomial representation). 
\end{lemma}

In particular the above lemma shows the existence of an efficient hitting set for sparse polynomials. We now give a lemma that shows the existence of an efficient hitting set for a product of sparse polynomials. 
Indeed it was shown in~\cite{ShpilkaVolkovich09} that if there is an efficient hitting set for any class of polynomials, then one can construct an efficient hitting set for a product of few polynomials from that class. Thus we immediately get the following lemma. 

\begin{lemma}[\cite{KlivansSpielman01,ShpilkaVolkovich09,SarafVolkovich11}]
\label{lem:pit-sparse-ks}
There exists a deterministic algorithm that given $n,s,d,k \in \N$  outputs a set $\SP_{(n,s,d,k)}$ of size $\poly(n,s,d,k)$ such that any set of (at most) $k$ non-zero $s$-sparse polynomials $f_1, \ldots, f_k \in \Fn$ with individual degrees at most $d$  have a common non-zero in $\SP_{(n,s,d,k)}$.  In other words, there exists $\ab \in \SP_{(n,s,d,k)}$ such that $\forall i: f_i(\ab) \neq 0$.
\end{lemma}

Note that in the above lemma, we could have replaced individual degree by total degree, and the result would have still held, since the total degree is at most a factor of $n$ more than the individual degree. However, in our applications, we will usually use an individual degree bound, and hence we stated the lemma in terms of individual degree.



As another simple corollary of \cite{KlivansSpielman01}, we obtain an efficient algorithm for sparse polynomial division, given an upper bound
on the sparsity of the quotient polynomial. In other words, if $f, g$ are sparse polynomials such that $f = g\cdot h$, then given black-box access to $f$ and $g$, one can recover $h$ (as long as it is also sparse). This is because given black-box access to $f$ and $g$, one can simulate black-box access to $h$. One can then use~\cite{KlivansSpielman01} to interpolate and recover $h$. If $h$ ends up being not sparse, then this algorithm would just reject. Moreover, given a candidate sparse polynomial $h$, it is easy to verify whether it is indeed the quotient polynomial of $f$ and $g$, but just multiplying out $h\cdot g$ and comparing with $f$.

\begin{lemma}[Corollary of \cite{KlivansSpielman01}]
\label{lem:sparse division}
Let $n,s,d,t \in \N$.
Let $f,g \in \Fn$ be $s$-sparse polynomials of degree at most $d$.
Then there exists an algorithm that given $f,g$ and $t$ uses $\poly(n,d,s,t)$ field operations
and computes the quotient polynomial of $f$ and $g$, if it is a $t$-sparse polynomial.
That is, if $f = g h$ for some $h \in \Fn$, $\mon{h} \leq t$, then the algorithm outputs $h$.
Otherwise, the algorithm rejects.
\end{lemma}

\newpage

\subsection{GCD and Resultants}
\label{sec:res}

Let  $f = a_d y^d + a_{d-1} y^{d-1} + \cdots + a_0$ and $g = b_e y^e + b_{e-1} y^{e-1} + \cdots + b_0$ be polynomials of $y$-degree exactly $d$ and $e$, respectively. Consider the $(d+e)\times (d+e)$ Sylvester Matrix whose first $e$ columns contain $e$ shifts of the vector of coefficients $(a_d,\ldots,a_0,0,\ldots,0)$, and next $d$ columns contain $d$ shifts of the vector of coefficients $(b_e,\ldots,b_0,0,\ldots,0)$. That is,

\begin{frame}
\footnotesize
\arraycolsep=1pt 
\medmuskip = 0.5mu 
\[
\Res_y(f,g) \,=\, \left|
\begin{array}{cccccccc}
a_d & 0 & \cdots & 0 & b_e & 0 & \cdots & 0 \\
a_{d-1} & a_d & \cdots & 0 & b_{e-1} & b_{e} & \cdots & 0 \\
a_{d-2} & a_{d-1} & \ddots & 0 & b_{e-2} & b_{e-1} & \ddots & 0 \\
\vdots & \vdots & \ddots & a_d & \vdots & \vdots & \ddots & b_e \\
\vdots & \vdots & \cdots & a_{d-1} & \vdots & \vdots & \cdots & b_{e-1} \\
a_0 & a_1 & \cdots & \vdots & b_0 & b_1 & \cdots & \vdots \\
0 & a_0 & \ddots & \vdots & 0 & b_0 & \ddots & \vdots \\
\vdots & \vdots & \ddots & a_1 & \vdots & \vdots & \ddots & b_1 \\
0 & 0 & \cdots & a_0 & 0 & 0 & \cdots & b_0 \\
\end{array}
\right|_{(d+e) \times (d+e)} \,
\]
\end{frame} 
\noindent This representation of resultant ensures that if $f$ and $g$ are sparse polynomials in $\Fn$ with small individual degree (in $y$), then sparsity of  $\Res_y(f,g)$ is bounded. We will use the following properties of resultant (For more info, see \cite[Chap.~7]{CGL92})  

\begin{lemma}[Resultant Properties]
\label{lem:res}
Let $f,g \in \Fny$ be monic in $y$, $s$-sparse polynomial with individual degrees at most $d$. Then:

\begin{enumerate}
\label{P1}
\item $\Res_y(f,g)(\xb)$ is an $ (2ds)^{2d}$-sparse polynomial over $\Fn$ with individual degrees at most $2d^2$.

\item For every $\ab \in \F^n$:
$\Res_y \left( f \restrict{\xb = \ab}  ,g \restrict{\xb = \ab} \right) = \Res_y (f,g) (\ab)$.

\item $\gcd(f,g) \neq 1$ iff $\Res_y(f,g) \equiv 0$.
\label{P3}

\end{enumerate}
\end{lemma}

\begin{definition}
For a field $\F$ we denote by $\cost_{\F}(d)$ the time of the best known algorithm that factors a univariate polynomial of degree $d$ over $\F$.
\end{definition}

\begin{lemma}[Univariate factoring] 
\label{lem:uni}
Let $f(x)\in \F[x]$ be a univariate polynomial of  degree $d$ then by the well known algorithms of Lenstra-Lenstra-Lovasz \cite{LLL82} and Berlekamp \cite{Berlekamp70, Shoup91,GathenGerhard99,GaoKL04}, $f$ can be factorized in time $\cost_{\F}(d)$. 
where: 
\begin{enumerate}
\item $\cost_{\F}(d) = \poly(\ell \cdot p,d)$, if $\F=\F_{p^\ell}$.
\item $\cost_{\F}(d) = \poly(d, t)$,  where $t$ is maximum bit-complexity of the coefficients of $f$, if $\F=\mathbb{Q}$.
\end{enumerate}
\end{lemma}

The next result which is implicit in many factorization algorithms, exhibits an efficient
factorization algorithm for certain regime of parameters. In particular, for polynomials with
constantly-many variables and a polynomial degree.  

\begin{lemma}[Implicit in \cite{Kaltofen89}, see also~\cite{Sudan98}]
There exists a deterministic algorithm that given a $r$-variate,
degree $d$ polynomial $f$ over $\F$ outputs its irreducible factors. The runtime of the algorithm is 
$(\cost_{\F}(d))^{\BigO(r)}$.
\end{lemma}

We will use this lemma for $r=2$ (i.e. bivariate factoring) in our deterministic factorization algorithm.

\section{Polytopes and Polynomials} 
\label{sec:Polytope}

In this section we will discuss various properties of polytopes, in particular the Newton polytope. These will be crucial ingredients in our proof of the sparsity bound for factors of sparse polynomials. The main results that we will discuss and develop are:
\begin{enumerate}
\item If $f,g,h$ are polynomials such that $f = g \cdot h$ then the sparsity of $f$ is lower bounded by $\max\set{\size{V(P_g)}, \size{V(P_h)}}$, where $P_g$ and $P_h$ are the Newton polytopes of $g$ and $h$ respectively, and where for a polytope $P$, $V(P)$ denotes the set of vertices of $P$. 
\item The convex hull of any subset of ${\set{0,1,\ldots, d} }^n$ must have ``many'' vertices (i.e. corner points). We will prove this as a corollary of an approximate version of Carath\'eodory's theorem. 
\end{enumerate}

Our approach to bounding the sparsity of factors of a polynomial using the theory of polytopes, and in particular Item 1 (as stated above) was inspired by a connection of the theory of polytopes to sparsity bounds that was observed by Dvir and Oliveira \cite{DvirOliveira14}. 
 
For a finite set of points $v_1, v_2, \ldots, v_k \in \R^n$, their \emph{convex span}, which we denote by $CS(v_1, \ldots, v_k)$ is the set defined by
$$CS(v_1, v_2, \ldots, v_k)= \condset{  \sum_{i= 1}^k \lambda_iv_i \,} { \, \lambda_i \geq 0 \mbox{ and } \sum_{i=1}^k \lambda_i = 1}.$$

A set $P \subseteq \R^n$ is a called a \emph{polytope} if there is a finite set of points $v_1, v_2, \ldots, v_k \in \R^n$ such that $P = CS(v_1, v_2, \ldots, v_k)$.
For a polytope $P$, and a point $a \in P$, we say that $a$ is a \emph{vertex} of $P$ if it \textbf{cannot} be written as $a = \lambda u + (1-\lambda) v$ for any $u, v \in P \setminus \{a\}$ and $\lambda \in [0,1]$. Alternatively, a vertex of $P$ is face of dimension $0$. We let $V(P)$ denote the set of vertices of $P$. 

It is an easy to verify, and a basic fact about polytopes, that if $P $ is a polytope, then $P = CS(V(P))$. Moreover, is $P = CS(v_1, v_2, \ldots, v_k)$ then $V(P) \subseteq \set{v_1, v_2, \ldots, v_k}$. 

(For more details see \cite{Ziegler12} Propositions 2.2 and 2.3)

\newpage

\subsection{The Newton Polytope and Minkowski Sum }

\begin{definition}
Given two polytopes $P_1$ and $P_2$ in $\R^n$, we define their \emph{Minkowski Sum} $P_1 + P_2$ to be the set of points given by $$P_1 + P_2 = \condset{v_1 + v_2} {v_1 \in P_1 \mbox{ and } v_2 \in P_2}.$$ 
\end{definition}

The following is a classic fact about the Minkowski sum of two polytopes. It basically says that the Minkowski sum of two polytopes is itself a polytope, and the number of vertices of each of the original polytopes is a lower bound on the number of vertices of the Minkowski sum. See~\cite{DvirOliveira14} (Theorem 3.12, Corollary 3.13), and~\cite{Schinzel00} for the formal details of a proof. After we state the result, we will provide an informal proof sketch which also gives some intuition for why the result holds. 

\begin{proposition} \label{main:polytope}
Let $P_1$ and $P_2$ be polytopes in $\R^n$. Then their Minkowski sum $P_1 + P_2$ is a polytope and $$\size{V(P_1 + P_2)} \geq \max\set{\size{V(P_1)}, \size{V(P_2)}}.$$
\end{proposition}

\begin{proof}[Proof sketch] Let $u_1, u_2, \ldots, u_{k_1}$ be the vertices of $P_1$ and $v_1, v_2, \ldots v_{k_2}$ be the vertices of $P_2$. Now, any element of $P_1 + P_2$ is of the form $\mu_1$ + $\mu_2$, where $\mu_1$ is a convex combination of $u_1, u_2, \ldots, u_{k_1}$ and $\mu_2$ is a convex combination of $v_1, v_2, \ldots v_{k_2}$. It follows easily from this that $\mu_1$ + $\mu_2$ is a convex combination of $V(P_1) + V(P_2) = \condset{ u+ v }{ u \in V(P_1), v \in V(P_2)}$.
Thus $P_1 + P_2 \subseteq CS(V(P_1) + V(P_2))$ and it is also easily to see that $CS(V(P_1) + V(P_2)) \subseteq P_1 + P_2$. Thus $P_1 + P_2 = CS(V(P_1) + V(P_2))$, and hence it is a polytope. 

We will now show that for every $u \in V(P_1)$, there exists $v \in V(P_2)$ such that $u+ v \in V(P_1 + P_2)$. 

Fix $u \in P_1$. Since $u \in V(P_1)$, there exists a hyperplane $H$ that passes through $u$ and such that all the rest of $P_1$ lies on one side of $H$. In particular there is a degree one polynomial $h$ such that $h(u) = 0$ and for every $u' \in P_1$ such that $u' \neq u$, $h(u') > 0$. (The hyperplane $H$ is the zero set of $h$.)
Moreover such an $h$ and $H$ can be chosen that are ``generic'' in the sense that none of the one-dimensional or higher faces of $P_1$ or $P_2$ can be translated to lie within $H$. (Such an $H$ can be obtained by doing a small random perturbation to the original $H$ about the point $u$.) 

Now for any real number $a$, consider the polynomial $h_a = h + a$. Let $H_a$ be the zero set of $h_a$. If $a$ is a large enough real valued number, then for any $v' \in P_2$, $h_a(v') >0$. Now slowly decrease the value of $a$ till for the first time, for some value $b$, $H_b$ touches $P_2$ at a single point, which will be some vertex $v$.  Since $H$ was a generic hyperplane, this we can ensure that $H_b$ only touches $P_2$ at a single point. 
Thus we will have the property that $h_b(v) = 0$ and for all $v' \in P_2$ such that $v' \neq v$, $h_b(v') > 0$. 

We will now show that for this choice of $v$, $u+v \in V(P_1 + P_2)$. 
Let $c$ be the constant term of $h$. Then $h = h' + c$, where $h'$ is a homogeneous degree one polynomial. Then $h_b$ = $h' + c + b$. 
Consider the degree one polynomial $h^* = h' + 2c + b$, and observe that $h^*(u+v) = h'(u+v) + 2c + b = (h'(u) + c) + (h'(v) + c + b) = 0 $. 
Moreover for any $u' \in P_1$, $v' \in P_2$ such that $(u,v) \neq (u', v')$, it must hold that $h^*(u'+v') > 0$. 
Thus it must be that $u+v \in V(P_1 + P_2)$. 

Observe also that the vertex $u+v$ of $P_1+P_2$ cannot be expressed as $u'+v'$ for any other $u' \in P_1$ and $v' \in P_2$ such that $(u',v') \neq (u,v)$. This is because $h^*(u+v) = 0$ but for  $(u',v') \neq (u,v)$, $h^*(u'+v') > 0$. 
Thus corresponding to the vertex $u \in P_1$, we have identified a vertex $u+v$ of $P_1+ P_2$ which cannot be expressed in any other way as a sum of a vertex of $P_1$ and a vertex of $P_2$. Since we can do this for each vertex of $P_1$, it follows that $\size{V(P_1)} \leq \size{V(P_1 + P_2)}$. By symmetry, $\size{V(P_2)} \leq \size{V(P_1 + P_2)}$, and the result follows. 

\end{proof}

For a polynomial $f \in \F[x_1, x_2, \ldots, x_n]$, suppose that $$f = \sum a_{i_1i_2 \ldots i_n}x_1^{i_1}x_2^{i_2}\cdots x_n^{i_n}.$$ 
For each coefficient $a_{i_1i_2 \ldots i_n} \neq 0$, we say that the exponent vector $(i_1, i_2, \ldots, i_n)$ is in the \emph{support} of $f$, when viewed as a vector in $\R^n$.  
We define $\supp(f)$ to be the set of all support vectors of $f$, i.e. $$\supp(f) = \condset{(i_1, i_2, \ldots, i_n) }{ a_{i_1i_2 \ldots i_n} \neq 0 }.$$

The convex hull of the set $\supp(f)$ is defined to be the {\it Newton polytope} of $f$, which we denote by $P_f$.

The following classic fact was observed by Ostrowski \cite{Ostrowski1921} in 1921. 
It states that if a polynomial $f$ factors as $g \cdot h$, then the Newton polytope of $f$ is the Minkowski sum of the Newton polytopes of $g$ and $h$. (See also~\cite{DvirOliveira14} (Proposition 3.16) for a proof.)  

\begin{proposition}
\label{prop:newton}

Let $f, g, h \in \Fn$ be polynomials such that $f = g \cdot h$. Then $$P_f = P_g + P_h.$$
\end{proposition}

\begin{remark}
This result will eventually play a crucial role in the proof of our sparsity bound. Note that we want to show that if a certain polynomial $f$ is sparse, then $g$ and $h$ are also sparse. We will show that if $g$ (or $h$) is ``dense", $f$ must also be ``dense". If we can show that $P_f$ has many vertices (i.e. corner points), then this will give us a lower bound on the number of monomials in $f$. Since $P_f = P_g + P_h$, a lower bound on $\size{V(P_g)}$ (or $\size{V(P_h)}$) is a lower bound on $\size{V(P_f)}$. Thus we then only need to lower bound $\size{V(P_g)}$, which we will show how to do using the results of the next section.  
\end{remark}

As an immediate corollary of the above two propositions, we easily recover the following basic bound relating the sparsity of polynomials to the Newton polytopes of its factors. (This bound was observed by Dvir and Oliveira in~\cite{DvirOliveira14}). 

\begin{corollary}
\label{cor:sparsity-newton} 
Let $f, g, h \in \Fn$ be polynomials such that $f = g \cdot h$. Then 
$$\mon{f} \geq \size{V(P_f)} \geq \max \set{\size{V(P_g)}, \size{V(P_h)}}.$$
\end{corollary}

\begin{proof}
By Proposition~\ref{prop:newton}, $P_f = P_g + P_h$, and hence by Proposition~\ref{main:polytope}, $$\size{V(P_f)} \geq \max\set{\size{V(P_g)}, \size{V(P_h)}}.$$

Since $P_f = \convhull(\supp(f))$, thus $V(P_f) \subseteq \supp(f)$. 
Hence $\size{V(P_f)} \leq \mon{f}$ and the result follows. 
\end{proof}

It is worth noting that if $d=1$, that is when $g$ ( or $h$) is multilinear, then every point in $P_g$ is a corner point. Hence,  $\size{V(P_g)}=\mon{g}$ and by Prop.~\ref{main:polytope} $\mon{f} \geq \mon{g}$.

\subsection{An approximate Carath\'eodory's Theorem}

Carath\'eodory's theorem is a fundamental result in convex geometry, and it states that if a point $\mu \in \R^n$ lies in the convex hull of a set $U$, then $\mu$ can be written as the convex combination of at most $n +1$ points of $U$. 

In order to prove our sparsity bound, we will be using an ``approximate" version of Carath\'eodory's theorem. The version that we use appears in~\cite{Barman15}\footnote{There is actually a small typo in the version of the theorem in~\cite{Barman15}, and the statement below fixes it.}. It essentially states that  if a point $\mu \in \R^n$ lies in the convex hull of a set $U$, then $\mu$ can be {\it uniformly} $\epsilon$-approximated in the $\ell_\infty$ norm by a vector that is the convex combination of only $\frac{\log n}{\epsilon^2}$ points of $U$. 

We first introduce some notation that we will use. For a set of vectors $U= \set{u_1, u_2, \ldots, u_m} \subseteq \mathbb R^n$, let $\convhull(U)$ denote the convex hull of $U$. (Note that for a finite set, the convex span of a set of vectors is the same as the convex hull of the vectors. Since in the rest of the paper we will only be dealing with finite sets, we will use the terms convex span and convex hull interchangeably).  A vector $\mu \in \convhull(U)$ is defined to be $k$-uniform with respect to $U$ if there exists a multiset $S$ of $[m]$ of size at most $k$ such that $\mu = \frac{1}{k}\sum _{i \in S} u_i$. 

 We present proof of the approximate Carath\'eodory theorem for completeness after stating the  theorem. There are other approximate versions of the Carath\'eodory Theorem that appear in the literature, often in terms of $\ell_p$ norms where $2 \leq p < \infty$. The version below is for the $\ell_{\infty}$ norm, and its proof is fairly straightforward.

\begin{theorem}[\cite{Barman15}, Theorem 3]
\label{thm:approx-caratheodory}

Given a set of vectors $U= \set{u_1, u_2, \ldots, u_m} \subseteq \mathbb R^n$ with $\max_{u \in U}  \left\|u\right\|_{\infty} \leq 1$, and $\epsilon > 0$. For every $\mu \in \convhull(U)$ there exists an $\BigO \left(\frac{\log n}{\epsilon^2}\right)$ uniform vector $\mu' \in \convhull(U)$ such that $\left\| \mu - \mu' \right\|_{\infty} \leq \epsilon$. 
\end{theorem}

\begin{proof}
Since $\mu \in \convhull(U)$, thus $\mu = \sum_{i=1}^m a_iu_i$, where for each $i \in [m]$, $a_i \geq 0$ and $\sum_{i=1}^m a_i = 1$.
Now consider the following probability distribution on $U$, where the probability of sampling $u_i$ is $a_i$. 
Pick $t = \left(\frac{\log n}{\epsilon^2}\right)$ samples independently from this distribution and let the resulting vectors be $v_1, v_2, \ldots, v_t$. 
Let $$\mu' = \frac{\sum_{i=1}^t v_i}{t}.$$
\begin{claim}
For any coordinate $j \in [n]$, $\Pr[ \abs{\mu_j - \mu'_j} > \epsilon ] < 1/n.$
\end{claim}
\begin{proof}
It follows immediately from the Chernoff-Hoeffding bounds applied to $t$ independent samples $Y_1, Y_2, \ldots, Y_t$ of the random variable $Y$, where for each $i \in [m]$, $Y = (u_i)_j$ with probability $a_i$. Then clearly $\E[Y] = \mu_j$, and $\mu'_j = \frac{\sum_{i=1}^t Y_i}{t}.$
Then by the Chernoff-Hoeffding inequality, $$\Pr[ |\mu_j - \mu'_j| > \epsilon ] < e^{-2\epsilon^2t} < 1/n.$$
\end{proof}

Once we have the claim, then a simple union bound over the coordinates implies that with positive probability,   $\left\| \mu - \mu' \right\|_{\infty} \leq \epsilon,$ and hence a suitable $\frac{\log n}{\epsilon^2}$ uniform vector $\mu' \in \convhull(U)$ exists. 
\end{proof}

\section{Sparsity Bound} 
\label{sec:sparsity}

In this section we prove the sparsity bound. 

\begin{theorem}[The Bound of Factor Sparsity]
\label{thm:sb}
There exists an non-decreasing function $\SB(n,s,d)\leq s^{\BigO({d^2\log{n}})}$ 
such that if $f \in \Fn$ is a polynomial of sparsity $s$ and individual degrees at most $d$, and if $f = g \cdot h$, for $g, h \in \Fn$, then the sparsity of $g$ is upper bounded by 
$\SB(n,s,d)$.
\end{theorem}

Before presenting the proof of the sparsity bound, we first show how to apply the approximate Carath\'eordory's Theorem~\ref{thm:approx-caratheodory} to show that the convex hull of any subset of ${\set{0,1,\ldots, d}}^n$ must have {\it many} vertices (i.e. corner points). 

\begin{theorem}\label{cor:cornerpoints}
Let $E \subseteq {\set{0,1,\ldots, d}}^n$. Let $t = \size{V(CS(E))}$. 
Then there exists an absolute constant $C$ such that 
Then $t^{C d^2 \log n} \geq \size{E}$. 
\end{theorem}

\begin{proof}
Let $E_d \subseteq [0,1]^n$ be the set obtained by taking $E$ and scaling every member of it down coordinate-wise by a factor of $d$. 
Let $\epsilon = 1/3d$. 
Let $U = V(CS(E)) \subseteq E$ be the set of vertices of the convex span of $E$. 
Similarly, let $U_d = V(CS(E_d)) \subseteq E_d$. 
Then clearly $\size{U} = \size{U_d}$. 

By Theorem~\ref{thm:approx-caratheodory}, for every $u_d \in E_d$, since $u_d \in \convhull(U_d)$, thus there exists an $\BigO\left(\frac{\log n}{\epsilon^2}\right)  = \BigO(d^2 \log n)$ uniform vector $u_d' \in \convhull(U_d)$ such that $\left\| u_d - u_d' \right\|_{\infty} \leq 1/3d$.

Observe that  for two distinct vectors $u_d, v_d \in E_d$, $\left\| u_d - v_d \right\|_{\infty} \geq 1/d$. 
Hence if $u_d' \in \convhull(U_d)$ is an  $\BigO(d^2 \log n)$ uniform vector such that 
$\left\| u_d - u_d' \right\|_{\infty} \leq 1/3d$
and if $v_d' \in \convhull(U_d)$ is an  $\BigO(d^2 \log n)$ uniform vector such that 
$\left\| v_d - v_d' \right\|_{\infty} \leq 1/3d$,
Then by the triangle inequality, we must have that $u_d' \neq v_d'$. 

The total number of $\BigO(d^2 \log n)$ uniform vectors that can be generated by the set $U_d$ is $\size{U_d}^{\BigO(d^2 \log n)}$. 
Moreover we have just shown that one can generate $\size{E_d}$ distinct  $\BigO(d^2 \log n)$ uniform vectors from $U_d$. 

Thus there is an absolute constant $C$ such that 
Thus, $$\size{U_d}^{Cd^2 \log n} \geq \size{E_d}$$

and we thus conclude that $t^{Cd^2 \log n} \geq \size{E}$.

\end{proof}

\begin{remark}
In fact, the dependence on $\log n$ in the theorem above is necessary. In particular, there is a set $E  \subseteq \{ -1,0,1\}^n$ such that the number of corner points in the convex hull of $E$ is $n$, but $\size{E} = n^{\Omega(\log n)}$.  However, it is not clear if such polytopes yield a polynomial with $\SB(n,s,d) =  s^{\Theta({d^2\log{n}})}$.
 An example of such a set (and the resulting polytope) was shared with us in \cite{Sap18}, and we describe it below. 
\end{remark} 

\newpage

\begin{claim}[\cite{Sap18}]
\label{example:Sap}
There is a set $E \subseteq  \{ -1,0,1\}^n$  s.t. $\size{V(CS(E))} = n$ and  $\size{E} = n^{\Omega(\log n)}$. 
\end{claim}

\begin{proof}
Let $m$ be a positive integer. Let $n = 2^m$ and let $H$ be the $n \times n$ Hadamard matrix. More precisely, $H$ is the matrix over the Reals with entries being $1$ or $-1$ such that if we index the rows and columns of $H$ by the elements of $\mb{F}_2^m$, and then the $(a,b)$ entry of $H$ is $(-1)^{\langle  a  , b\rangle }$, for all $a,b \in \mb{F}_2^m$. 

Let $V\subseteq \{-1, +1\}^n$ be the set of column vectors of $H$. We will show that the convex span of $V$ contains at least $n^{\Omega(\log n)}$ distinct elements of $\{ -1,0,1\}^n$, and this will suffice to prove the claim. 

Recall that each element of $V$ is indexed by an element of $\mb{F}_2^m$.

We will show that for each  $S \subseteq \mb{F}_2^m$ that is a linear subspace, if we take the uniform convex span of the elements of $V$ that correspond to the elements of $S$, then we get an element of $\{ 0,1\}^n$, Moreover, distinct subspaces give rise to distinct elements of $\{ 0,1\}^n$. Since the number of subspaces is $n^{\Omega(\log n)}$, the result then follows. 

Now, let $S \subseteq \mb{F}_2^m$ be a linear subspace, and let $u_S$ be the characteristic vector of this subspace. We need to show that, $$\frac{1}{\size{S}} H\cdot u_S \subseteq \{0,1\}^n.$$ 

Let $T=\{b \in \mb{F}_2^m : \langle  a,  b\rangle=0,  \, \, \forall \, a \in S\}$.
With slight abuse of notation let $(H \cdot {u}_S)_{b} $ corresponds to the entry in the b-th coordinate. Notice that, if $b \in T$, then $(H \cdot {u}_S)_{b}=\sum_{a\in S}(-1)^{\langle a, b \rangle } =\size{S}$.

On the other hand, if $b \not\in T$, then $\langle a_o,b \rangle  =1 $ for some $a_o \in S$. Thus, for each $a \in S$, we have that $(-1)^{\langle  a+ a_o , b\rangle }$ and $(-1)^{\langle  a,b \rangle }$ to have different signs. Hence, $(H \cdot {u}_S)_{b} =0$ in this case.

Thus,
$$\frac{1}{\size{S}}H \cdot {u}_S \subseteq \{ 0,1\}^n,$$ and the coordinates that equal $1$ are precisely those that correspond to the orthogonal subspace of $S$. Thus distinct subspaces $S$ give rise to distinct vectors in $\{ 0,1\}^n$.
\end{proof}

\begin{remark}
In order to obtain a polytope with non-negative coordinates, one can simply shift all the coordinates by $1$. 
\end{remark} 

We now prove Theorem~\ref{thm:sb}.

\begin{proof}[Proof of Theorem~\ref{thm:sb}]

Let $\mon{g}$ denote the sparsity of $g$. Thus $g$ has $\mon{g}$ monomials.
Let $\supp(f), \supp(g) \subseteq {\set{0,1,\ldots, d}}^n$ denote the sets of exponent vectors of $f$ and $g$, respectively. 

Let $t_g =\size{V(CS(\supp(g)))}$. Thus $t_g$ denotes the number of vertices of the polytope which is the convex span of $\supp(g)$. Similarly let $t_f =\size{V(CS(\supp(f)))}$.
 By Theorem~\ref{cor:cornerpoints},
$$t_g\geq \mon{g}^{\frac{1}{Cd^2 \log n}}.$$
Now, by Corollary~\ref{cor:sparsity-newton}, $t_g \leq t_f$. Moreover, since $ V(CS(\supp(f))) \subseteq E_f$, thus $t_f \leq \size{E_f}$, which equals the sparsity of $f$, $\mon{f}$. 
Hence $$\mon{f} \geq \mon{g}^{\frac{1}{Cd^2 \log n}}$$
and the theorem follows.

\end{proof}

\subsection{Tightness of Sparsity Bound}
\label{sec:tightness}

In this section we see some examples of polynomials that have factors with a significantly larger number of monomials then the original polynomials. In the case where we do not bound the individual degree of the polynomials, the factors can have a superpolynomial number of monomials. 

Interestingly, the examples we will see are also tight cases of Prop.~\ref{main:polytope}.

The following example was noted by von zur Gathen and Kaltofen~\cite{GathenKaltofen85}. 

\begin{example}[\cite{GathenKaltofen85}] 
\label{eg1}
Let \begin{align*}
f(x)=&\prod_{i=1}^n (x_i^d-1), \\
g(x)=&\prod_{i=1}^n (1+x_i+\ldots+x_i^{d-1})
\end{align*}
\end{example}

Notice that, $g$ is a factor of $f$, but $\mon{f}=2^n$ and $\mon{g}=d^n$. Thus letting $s$ denote the sparsity of $f$, notice that $\mon{g}=s^{\log{d}}$, where $d$ is the individual degree of $f$. 

Indeed, for fields of characteristic $0$, this is the best ``blow-up'' of the sparsity that we are aware of. 

Our next example works for fields of positive characteristic, say $\mb{F}_p$, and uses the Frobenius action of powering by $p$. In this example we see a much bigger ``blow-up'' than in the previous example.     
\begin{example} \label{eg2}
Let $f \in \F_p[x_1,\ldots x_n]$, $p$-prime and let $0 < d < p$.  \begin{align*}
f(x)=& x_1^p + x_2^p + \ldots +x_n^p , \\
g(x)=&\big(x_1+x_2 + \ldots x_n \big)^{d}
\end{align*}
\end{example}

Notice again that $g$ is a factor of $f$, but $\mon{f}=n$ and $\mon{g}={{n+d-1}\choose{d}} \approx n^d$.  Thus if $s$ denotes sparsity of $f$, then $\mon{g}=s^{d}$. 

The above example is particularly interesting because it shows that for general sparse polynomials, with no bound on the individual degree (for instance if the individual degree can be as large as $n$), the factors of a polynomial can have exponentially more monomials than the original polynomial! Thus there is no hope of proving an efficient sparsity bound for general sparse polynomials. 

Note however, that this example only applies to fields of certain characteristics. For instance for fields of characteristic $0$, the previous example might be the one with the worst possible blowup, and hence for such fields, an efficient sparsity bound for general sparse polynomials might still hold. However any proof of such a sparsity bound must be able to take advantage of the properties of the underlying field. The techniques for the sparsity bound proved in this paper are oblivious to the underlying field, and thus, given Example~\ref{eg2}, the best possible sparsity bound for factors of a polynomial that one can hope to show with such techniques is of the form $\BigO(s^{d})$.

\newpage
\section{Factoring Algorithm} 
\label{sec:factAlgo}

In this section, we give our deterministic factorization algorithm for sparse polynomials with small individual degree, thus proving Theorem \ref{THM:Main2}. The runtime of the algorithm strongly depends on the bound in Theorem \ref{thm:sb}. To emphasize this dependence, we state our results in terms of $\SB(n,s,d)$. Theorem \ref{THM:Main1} follows by instating the upper bound.  

As outlined in Section \ref{sec:proof 2}, we first focus on monic polynomials. Then we show how to extend the algorithm to general polynomials. 


\subsection{Black-box Factoring of Sparse Monic Polynomials (given some advice)} 
\label{sec:BBfact}

In this section we give an algorithm that takes as input a sparse monic polynomial $f(y, \xb)$ of bounded individual degree, as well as some additional information about its factorization pattern, and then outputs (in some sense) blackbox access to its factors. 

The algorithm mirrors that black-box factorization algorithm of \cite{KaltofenTrager90}. 

The algorithm assumes that it is given an assignment $\ab \in \F^n$ for which no two distinct coprime factors of $f(y,\xb)$ have non-trivial gcd, when we set $\xb = \ab$,
and it is given the correct partition of the factors of $f(y,\ab)$ (i.e. the partition gives the grouping of the factors of $f(y,\ab)$ that will correspond to the factors of $f$). The algorithm outputs evaluations of the irreducible factors of $f$ at any input $(y_0, \bb) \in \F^{n+1}$. More precisely, for any $\bb \in \F^n$, and any irreducible factor $h_i(y, \xb)$ of $f$, the algorithm will output the univariate polynomial $h_i(y, \bb)$ which can then be evaluated at $y_0$. 

Given an input point $\bb \in \F^n$, the algorithm uses $\ab$ as an anchor point and draws a line to $\bb$. Next, the algorithm computes a bi-variate factorization of 
the polynomial $f(y, \ell_{\ab,\bb}(t))$ (see Definition~\ref{def:line}). Finally, the algorithm outputs the black-boxes for each factor of $f$ by matching the factors of 
$f(y, \ell_{\ab,\bb}(t))$ to the factors of $f(y,\ab)$. We will describe our black-box factoring algorithm below:

\begin{algorithm}[H] 
\label{alg:black-box factorization}
    \KwIn{$s$-sparse monic (in $y$) polynomial $f \in \Fny$ of individual degree at most  $d$. \\
    Assignments: $\ab, \bb \in \F^n$ \\ 
    Univariate Polynomials: $g_1(y), g_2(y), \ldots , g_r(y)$ \\
    Partition: $A_1 \dot{\bigcup} A_2 \dot{\bigcup} \cdots 
    \dot{\bigcup} A_m = [r]$ \\ 
    Exponent Vector: $\bar{e} =  (e_1, e_2, \ldots, e_m) \in [d]^{m}$}
\KwOut{Univariate Polynomials: $\varphi_1(y), \varphi_2(y), \ldots, \varphi_m(y)$}

    \bigskip    
    $\tilde{f}(y,t) \leftarrow f(y, \ell_{\ab,\bb}(t))$;
    
    Compute the bi-variate factorization of $\tilde{f}(y,t) = f_1^{v_1}(y,t) \cdot f_2^{v_2}(y,t) \cdots f_{r'}^{v_{r'}}(y,t)$;
    \tcc{wlog the polynomials are monic in $y$}
    
    
     \For{$i \leftarrow 1$ \KwTo $m$}{
   
     $\tilde{h}_i(y,t) \leftarrow 1$;
    
     \For{$k \leftarrow 1$ \KwTo $r'$}{
    
      \lIf{there exists $j \in A_i$ s.t $g_j(y) \divs f_k(y,0)$}
       {$\tilde{h}_i(y,t) \leftarrow \tilde{h}_i(y,t) \cdot f_k^{v_k / e_i}(y,t)$}{} 
     }
     }   
     \Return $\tilde{h}_1(y,1), \tilde{h}_2(y,1), \ldots, \tilde{h}_m(y,1)$;
     \caption{Black-Box Evaluation of Factors}
\end{algorithm}


\begin{lemma}[Black-box Factorization]
\label{lem:black-box factorization}
Let $f(y,\xb) \in \Fny$ be a polynomial monic in $y$ with individual degrees at most $d$. 
Suppose $f(y,\xb)$ can be written as
$f(y,\xb) = \prod \limits _{i=1}^m h_i^{e_i}(y,\xb)$ 
such that $\gcd(h_i,h_{i'}) = 1$ for $i \neq i'$.
Then given:
\begin{enumerate}
\item 
\label{C1}
a point $\ab \in \F^n$ such that $\forall i \neq i': \Res_y(h_i, h_{i'})(\ab) \neq 0$

\item 
\label{C2}
monic irreducible polynomials $g_1(y), g_2(y), \ldots , g_r(y)$

\item  
\label{C3}
a partition $A_1 \dot{\bigcup} A_2 \dot{\bigcup} \cdots 
    \dot{\bigcup} A_m = [r]$ such that for all $i \in [m]: h_i(y,\ab) = \prod \limits _{j \in A_i} g_j(y)$

\item
\label{C4}
exponent vector $\bar{e} =  (e_1, e_2, \ldots, e_m) \in [d]^{m}$ 
\end{enumerate}
and a point $\bb \in \F^n$,  Algorithm \ref{alg:black-box factorization} computes $h_{1}(y,\bb), h_{2}(y,\bb), \ldots, h_{m}(y,\bb)$, using $\poly(n,\cost_{\F}(d))$ field operations.
 


\end{lemma}
\begin{proof}
We claim that for each $i \in [m]: \tilde{h}_{i}(y,t) = h_i(y,\ell_{\ab,\, \bb}(t))$ and hence $$\tilde{h}_{i}(y,1) = h_i(y,\ell_{\ab,\, \bb}(1)) = h_i(y,\bb).$$

Since, $f(y,\xb) = \prod \limits _{i=1}^m h_i^{e_i}(y,\xb)$, substituting 
$\xb=\ell_{\ab,\, \bb}(t)$ implies that $\tilde{f}(y,t) = \prod \limits _{i=1}^m h_i^{e_i}(y,\ell_{\ab,\, \bb}(t))$. Fix $k \in [r']$. By Uniqueness of Factorization (Lemma \ref{lem:ufd}), 
$$\exists i: f_k(y,t) \divs h_i(y,\ell_{\ab,\, \bb}(t)).$$
Thus setting $t= 0$, we get that 
$$f_k(y,0) \divs h_i(y,\ab).$$
Now since $h_i(y,\ab)$ is a product of irreducible polynomials $g_j(y)$, for $j \in A_i$, thus 
$$\exists j \in A_i: g_j(y) \divs f_k(y,0).$$

The last step follows from Pre-condition \ref{C3} and Uniqueness of Factorization (Lemma \ref{lem:ufd}). Now suppose $\exists i'$ and  $j' \in A_{i'}: g_{j'}(y) \divs f_k(y,0)$. It follows that $g_{j'}(y) \divs h_i(y,\ab)$. 

Thus, $ \gcd \left( h_i(y,\ab) , h_{i'}(y,\ab) \right) \neq 1.$
By Lemma \ref{lem:res}:
$\Res_y(h_i, h_{i'})(\ab) = 
\Res_y \left( h_i \restrict{\xb = \ab} , h_{i'} \restrict{\xb = \ab} \right) = 0.$
By  Pre-condition \ref{C1}, this is only possible if $i=i'$. 

Now, fix $i$. Let $f_k(y,t)$ and $u_k$ be an irreducible factor of $h_i(y,\ell_{\ab,\, \bb}(t))$ and its degree in the latter, respectively. Now observe that $f_k(y,t)$ doesn't divide any other 
 $h_{i'}(y,\ell_{\ab,\, \bb}(t))$ : suppose it did, then setting $t=0$ and repeating the previous argument we would get a contradiction.
 Consequently, $v_k = u_k \cdot e_i$ and hence: $$f_k^{u_k}(y,t) =
f_k^{v_k / e_i}(y,t) $$ and $$f_k^{v_k / e_i}(y,t)\divs \tilde{h}_i(y,t).$$
We conclude that $h_i(y,\ell_{\ab,\, \bb}(t)) \divs \tilde{h}_{i}(y,t)$. The claim follows by observing that :
$$\prod \limits _{i=1}^m \tilde{h}_i^{e_i}(y,t) =
f_1^{v_1}(y,t) \cdot f_2^{v_2}(y,t) \cdots f_{r'}^{v_{r'}}(y,t) =
\tilde{f}(y,t) = \prod \limits _{i=1}^m h_i^{e_i}(y,\ell_{\ab,\, \bb}(t)).$$
For the runtime, observe that $m,r,r' \leq d$ and clearly $d = \BigO(\cost_{\F}(d))$.
\end{proof}

\subsection{Factoring Sparse Monic Polynomials (without advice)}

With the black-box factoring algorithm of the previous subsection, we get blackbox access to the irreducible factors of the input monic sparse polynomial, and we can use a reconstruction algorithm to reconstruct the actual factors. The caveat is that black-box factorization algorithm of the previous section assumes that it is given some additional information: an assignment $\ab \in \F^n$ for which no two distinct factors of $f(y,\xb)$ have non-trivial gcd, when we set $\xb = \ab$, and the correct partition of the factors of $f(y,\ab)$. 

In this section we show that the advice is actually a member of a small set that can be computed, and hence one can just ``guess" the advice!
Since $f(y,\ab)$ has at most $d$ factors, the number of possible partition is $d^{\BigO(d)}$.
Hence we can ``guess" the correct partition by trying out all the possibilities. In terms of finding $\ab$ as above, the following lemma shows that there exists a small set of points $S \subseteq \F^n$ that contain a point $\ab$ with the required properties for every monic sparse polynomial of degree $d$.


\begin{lemma}
\label{lem:sparse-res}
Let $f \in \Fny$  be monic in $y$, $s$-sparse polynomial with individual degrees at most $d$ and let $f(y,\xb) =  h_1^{e_1}(y,\xb)  \ldots  h_k^{e_k} (y,\xb)$ be the unique monic factorization of $f(y,\xb)$.
Then there exists a set $S$ of size $\size{S} = (n \cdot \SB(n,d,s))^{\BigO(d)}$ such that for any $f$ as above
there exists an assignment $\ab \in \F^n$ satisfying $\forall i \neq i': \Res_y(h_i, h_{i'})(\ab) \neq 0$.
\end{lemma}

We defer the proof of the lemma to the end of the section.

Given a polynomial $f$, the algorithm will proceeding by trying to reconstruct the factors of $f$ for every projection in $S$ and every partition. Given a purported factorization, we can test it by explicitly multiplying out the polynomials. Clearly, the algorithm will pick up \emph{any} valid factorization of $f$ (not just the irreducible one). We will select the irreducible factorization using the simple characterization given in Lemma \ref{lem:max fact}.

\begin{algorithm}[H] 
\label{algo:factor monic sparse} 

 \KwIn{$s$-sparse polynomial $f \in \Fny$, monic in $y$, of individual degree at most  $d$.}
    \KwOut{monic irreducible factors  $h_{1}, h_{2}, \ldots, h_{m}$, and $e_1, e_2, \ldots, e_m$ such that $f=h_1^{e_1} \cdots h_m^{e_m}$ 
    }
     
    \bigskip

\begin{enumerate}

\item For each $\ab \in S$ (from Lemma \ref{lem:sparse-res}), subset $I \subseteq [d]$, $m' \in [d]$,
a non-empty \textit{partition} of $I$: $A_1 \dot{\bigcup} A_2 \dot{\bigcup} \cdots \dot{\bigcup} A_{m'} = I$,
and exponent vector $\bar{e}' = (e'_1, e'_2, \ldots, e'_{m'}) \in [d]^{m'}$:

\begin{enumerate}

\item Compute the monic univariate factorization  $f(y,\ab) = \prod \limits_ {j=1}^rg_j(y)$  (Using Lemma \ref{lem:uni})

\item Call Algorithm \ref{alg:black-box factorization} with $f, \ab, \set{A_i}_{i \in [m']}, \bar{e}$ and $\set{g_j(y)}_{j \in I}.$

\item Invoke the reconstruction algorithm from Lemma \ref{lem:rec}
with $n' = n, s' = \SB(n,d,s), d' = d$ using the above as an oracle to reconstruct the polynomials $h'_1(y,\xb), \ldots h'_m(y,\xb)$.


\item Test that $f \equiv {h'}_1^{e'_1} \cdot {h'}_2^{e'_2} \cdots {h'}_{m'}^{e'_{m'}}$
factorization. (Via explicit multiplication)

\end{enumerate}

\item Return a factorization that maximizes the expression $\Phi(\bar{e}) \eqdef 2 \cdot \sum \limits_{i=1}^{m'} {e'}_i - {m'}$.  \tcc{Pick the most ``refined'' factorization}

\end{enumerate}
\caption{Sparse Monic Polynomial Factorization}
\end{algorithm}

\begin{lemma}
\label{lem:monic sparse}
Let $f(y,\xb) \in \Fny$ be a polynomial, monic in $y$, with individual degrees at most $d$. Given $f$, Algorithm \ref{algo:factor monic sparse} computes the unique monic factorization of $f$.
That is, the algorithm outputs coprime, monic irreducible polynomials $h_{1}, h_{2}, \ldots, h_{m}$, and $e_1, e_2, \ldots, e_m$ such that $f=h_1^{e_1} \cdots h_m^{e_m}$, using at most $(n \cdot \SB(n,d,s))^{\BigO(d)} \cdot\poly(\cost_{\F}(d))$ field operations.
\end{lemma}

\begin{proof}
Let $f=h_1^{e_1} \cdots h_m^{e_m}$ be the unique monic factorization of $f$. By definition, $\forall i \neq i': \gcd(h_i, h_{i'}) = 1$,
$\forall i: e_i \leq d$ and $m \leq d$.  
We first claim that as the algorithm iterates over all settings of $\ab, I, m', \bar{e}'$ and the partition, one of these ``guesses" satisfies the pre-conditions of Lemma \ref{lem:black-box factorization}.

By Lemma \ref{lem:sparse-res}, there exists $\ab \in S$ (where $S$ is the set from Lemma \ref{lem:sparse-res}) such that $\forall i \neq i': \Res_y(h_i, h_{i'})(\ab) \neq 0$.
Consider the monic univariate factorization:
$$\prod \limits_ {j=1}^r g_j(y) = f(y,\ab) =  \prod \limits_ {i=1}^m h_i^{e_i}(y,\ab).$$ 
Clearly, $r \leq d$, and the claim follows from the uniqueness of factorization (Lemma \ref{lem:uni}). Therefore, by Lemma \ref{lem:black-box factorization}, given this guess, Algorithm \ref{alg:black-box factorization} will produce oracle access for the polynomials $h_1,\cdots, h_m$. By Theorem \ref{thm:sb}, $\forall i: \mon{h_i} \leq \SB(n,s,d)$. Therefore, the reconstruction algorithm from Lemma \ref{lem:rec} will, indeed, output the polynomials $h'_1, \ldots, h'_m$ such that $\forall i: h'_i \equiv h_i$, which will pass the subsequent tests. 

Let $f \equiv {h'}_1^{e'_1} \cdot {h'}_2^{e'_2} \cdots {h'}_{m'}^{e'_{m'}}$ be the factorization returned by the algorithm.
By Lemma \ref{lem:max fact}, the polynomials ${h'}$ are irreducible and pairwise coprime. The final claim follows by uniqueness of factorization.


\noindent \textbf{Runtime Analysis:}
By Lemma \ref{lem:sparse-res}, there are $\poly(n \cdot \SB(n,d,s))^{\BigO(d)} \cdot d^{\BigO(d)}$ iterations. We outline the runtime of each step in a iteration:

\begin{enumerate}

\item By Lemma \ref{lem:uni} - $\cost_{\F}(d)$.

\item By Lemma \ref{lem:black-box factorization} - $\poly(n,\cost_{\F}(d))$ per query.

\item By Lemma \ref{lem:rec} - $\poly(n,\SB(n,d,s),d)$ time and queries.


\item By Theorem \ref{thm:sb} -  $\SB(n,d,s)^{\BigO(d)}$.
\end{enumerate}
Putting all together: $n^{\BigO(d)} \cdot \SB(n,d,s)^{\BigO(d)} \cdot
\poly(\cost_{\F}(d))$.

\end{proof}

We now give the proof of Lemma \ref{lem:sparse-res}.

\begin{proof}[Proof of Lemma \ref{lem:sparse-res}]
By Theorem \ref{thm:sb}, for each $i \in [k]: \mon{h_i} \leq \SB(n,d,s)$. For $i,i' \in [k]$, consider the polynomials: $$f_{(i,i')}(\xb) \eqdef \Res_y (h_i,h_{i'})(\xb).$$ Fix $i, i' \in [k]$ such that $i \neq i'$.
By definition, $f_{(i,i')} \nequiv 0$. Moreover, by Lemma \ref{P1}, $f_{(i,i')}$ is $(2d \cdot \SB(n,d,s))^{2d}$-sparse polynomial with individual degrees at most $2d^2$. 
As $k \leq d$, by Lemma \ref{lem:pit-sparse-ks}, 
$\SP_{\left( n, \; (2d \cdot \SB(n,d,s))^{2d}, \; 2d^2, \; d^2 \right)}$ contains a common non-zero for all $f_{(i,i')}$-s. 
The claim about the size follows from Lemma \ref{lem:pit-sparse-ks}. 
\end{proof}



\subsection{Factoring General Sparse Polynomials}

In this section we show how to extend the factorization algorithm for monic sparse  polynomials to general sparse polynomials. We begin by showing how to convert a (general)  sparse polynomial with ``small'' individual degrees into a ``somewhat'' sparse monic polynomial of a ``slightly larger'' individual degrees. 

\begin{definition}
\label{def:make monic}
Let $f(x_1, \ldots, x_n, x_{n+1}) \in \Fna$ and let $k \leq d$ denote the degree of $x_{n+1}$ in $f$. Let $f_{k} \eqdef \lc_{x_{n+1}}(f)$. We define: $\hat{f}(y, x_1, \ldots, x_n) \eqdef 
f_{k}^{k-1} \cdot f(x_1, \ldots ,x_n, \frac{y}{f_k}).$
\end{definition}

\begin{lemma}
\label{lem:monic sparsity and degree}
Suppose $f$ is an $s$-sparse polynomial with individual degrees at most $d$. Then function $\hat{f}$ is an $(s^d)$-sparse polynomial in $\Fny$, monic in $y$ with individual degrees at most $d^2$.
\end{lemma}

\begin{proof}
Write:  $f = \sum _{j=0}^k f_j \cdot x_{n+1}^j$ such that $\forall j, x_{n+1} \not \in \var(f_j).$ Then
$$\hat{f} = \sum \limits _{j=0}^k f_j \cdot f_k^{k-1} \cdot y^j / f_k^j = y^k + \sum \limits _{j=0}^{k-1} f_j \cdot f_k^{k-1-j} \cdot y^j.$$
Observe that for every $x_i: 
\deg_{x_i}(f_j \cdot f_k^{k-1-j}) \leq d + d(k-1) \leq d^2$.
For the sparsity of $\hat{f}:$
$$\mon{\hat{f}} = 1 + \sum \limits _{j=0}^{k-1} \mon{f_j \cdot f_k^{k-1-j}} 
\leq \sum \limits _{j=0}^{k} \mon{f_j} \cdot \mon{f_k}^{k-1} 
\leq \sum \limits _{j=0}^{k} \mon{f_j} \cdot s^{k-1} = \mon{f} \cdot s^{k-1} \leq s^k \leq s^d.$$
\end{proof}

In addition to the question regarding the sparsity of the polynomial $\hat{f}$, there are two follow-up questions we need to address: 
\begin{enumerate}
\item How are the factors of $\hat{f}$ related the original factors of $f$? 

\item As the degree of $y$ in $\hat{f}$ is at most $d$, we can recover at most $d$ factors, while $f$ could potentially have $dn$ factors!
How can we recover the remaining factors?
\end{enumerate}

The following lemma provides the answers to both questions.

\begin{lemma}
\label{lem:monic factorization}
Let $f(\xb,x_{n+1}) = \prod \limits _{i=1}^{m'} h_i^{e_i}(\xb,x_{n+1}) \cdot \prod \limits _{l=m'+1}^{m} h_l^{e_l}(\xb)$ and $f_k(\xb) = \prod \limits _{j=1}^r w_j^{\beta_j}(\xb)$ be pair-wise coprime, irreducible factorizations of $f$ and $f_k$, respectively such that $x_{n+1} \in \var(h_i)$ iff $i \in [m']$.
Furthermore, let $\hat{f}(y,\xb) = \prod \limits _{j=1}^{\hat{m}} \hat{h}_j^{\hat{e}_j}(y,\xb)$ be the unique monic factorization of $\hat{f}$. Then 

\begin{enumerate}
\item $\hat{m} = m'$ and there exist polynomials $u_1(\xb), \ldots, u_{m'}(\xb) \in \Fn$ and a permutation $\sigma:[m'] \to [m']$ such that: $\hat{h}_i \left(f_k \cdot x_{n+1},\xb \right) = h_{\sigma(i)} \left( \xb,x_{n+1} \right) \cdot u_i(\xb)$  and $\hat{e}_i = e_{\sigma(i)}$ for $i \in [m']$.

\item $m - m' \leq r$. Moreover, there exists an injective map $\tau : \set{m'+1, \ldots, m} \to [r]$ such that $h_l$ and $w_{\tau(l)}$ are nonzero scalar multiples of each other (i.e. $h_l \sim w_{\tau(l)})$, for $l \in \set{m'+1, \ldots, m}$.
\end{enumerate}
\end{lemma}

We defer the proof of the lemma to the end of the section.


In light of the above, the algorithm proceeds by first converting a given polynomial $f$ into a monic polynomial $\hat{f}$ to recover the factors that depend on $x_{n+1}$. Next, the algorithm recursively factors $f_k$ (that does not depend on $x_{n+1}$) to recover the factors that \emph{do not} depend on $x_{n+1}$ (if any). 

\begin{algorithm}[H] 
\label{alg:factor sparse}

 \KwIn{$s$-sparse polynomial $f \in \Fna$ with individual degrees at most  $d$.}
    \KwOut{irreducible factors  $h_{1}, h_{2}, \ldots, h_{m}$, and $e_1, e_2, \ldots, e_m$ such that $f=h_1^{e_1} \cdots h_m^{e_m}$ 
    }
     
    \bigskip 

\begin{enumerate}
\item \lIf{$n \leq 1$}{Return the bi-variate factorization of $f$}

\item $k = \deg_{x_{n+1}}(f)$; $f_{k} \leftarrow \lc_{x_{n+1}}(f)$;

\item Compute $\hat{f}(y,\xb)$ (Using Definition \ref{def:make monic})

\item Compute the unique monic factorization $\hat{f}(y,\xb) = \prod \limits _{i=1}^{\hat{m}} \hat{h}_i^{e_i}(y,\xb)$ (Using Algorithm \ref{algo:factor monic sparse})

\item 
\label{Step restore}
\lForEach{$i \in [\hat{m}]$}{$h_i(x_1, \ldots, x_n, x_{n+1}) \leftarrow \hat{h}_i (f_{k} \cdot x_{n+1},\xb)$}

\item Recursively compute a factorization of $f_{k}(x_1, \ldots, x_n) = \prod \limits _{j=1}^r w_j^{\beta_j}(x_1, \ldots, x_n)$

\item \For{$j \leftarrow 1$ \KwTo $r$}{
$\alpha_j \leftarrow -\beta_j \cdot ( k -1 )$;

\For{$i \leftarrow 1$ \KwTo $\hat{m}$}{
Find the maximal $d_{ij}$ such that $w_j^{d_{ij}} \divs h_i$;
\tcc{By iteratively applying Lemma \ref{lem:sparse division} with $t = \SB(n,d^2,s^d)$}

$\alpha_j \leftarrow \alpha_j + {d_{ij}} \cdot e_i$; $h_i \leftarrow h_i / w_j^{d_{ij}}$; } 
}

\item \Return $h_1, \ldots, h_{\hat{m}}, w_1, \ldots, w_r$ and $e_1, \ldots, e_{\hat{m}}, \alpha_1, \ldots, \alpha_r$; \tcc{Return only those where $\alpha_j > 0$}

\end{enumerate}
\caption{Main Algorithm: overview}
\end{algorithm}

\begin{theorem}
\label{thm:fact sparse}
Let $f(\xb) \in \Fn$ be a polynomial with individual degrees at most $d$. Given $f$, Algorithm \ref{alg:factor sparse} outputs pairwise coprime, irreducible polynomials $h_{1}, h_{2}, \ldots, h_{m}$, and $e_1, e_2, \ldots, e_m$ such that $f=h_1^{e_1} \cdots h_m^{e_m}$, using  
$\left( n \cdot \SB(n,d^2,s^d) \right)^{\BigO(d^2)} \cdot\poly(\cost_{\F}(d^2))$ field operations.
\end{theorem}

\begin{proof}
The correctness of the algorithm follows from Lemmas \ref{lem:monic sparsity and degree} and \ref{lem:monic factorization}. In particular,
by Theorem \ref{thm:sb}: $\mon{\hat{h}_i}, \mon{h_i} \leq \SB(n,d^2,s^d)$.

\noindent \\ \textbf{Runtime Analysis:}
 Let $T(n,s,d)$ denote the number of field operations of the algorithm given an $s$-sparse, $n$-variate polynomial of individual degrees at most $d$.
We get that $T(1,s,d), T(2,s,d) = \poly(\cost_{\F}(d))$. For $n\geq 3$, we outline the runtime of each step.

\begin{enumerate}

\item $T(1,s,d)$,  $T(2,s,d) = \poly(\cost_{\F}(d))$.

\item $\poly(n,s,d)$.

\item By Lemma \ref{lem:monic sparsity and degree} - $\poly(n,s^d)$.

\item By Lemmas \ref{lem:monic sparsity and degree} and
\ref{lem:monic sparse} - 
$(n \cdot \SB(n,d^2,s^d))^{\BigO(d^2)} \cdot\poly(\cost_{\F}(d^2))$.

\item By Lemma \ref{lem:monic sparsity and degree} and Theorem \ref{thm:sb} - $\poly \left( \SB(n,d^2,s^d) \right)$.

\item Since $\mon{f_k} \leq \mon{f} \leq s$ - $T(n-1,s,d)$.

\item By Lemma \ref{lem:sparse division} - $\poly \left(nd , \SB(n,d^2,s^d) \right)$.
\end{enumerate}

Consequently: $T(n,s,d) = T(n-1,s,d) + 
\left( n \cdot \SB(n,d^2,s^d) \right)^{\BigO(d^2)} \cdot\poly(\cost_{\F}(d^2))$ implying that 
$T(n,s,d) = \left( n \cdot \SB(n,d^2,s^d) \right)^{\BigO(d^2)} \cdot\poly(\cost_{\F}(d^2))$.
\end{proof}

We now give the proof of Lemma \ref{lem:monic factorization}.

\begin{proof}[Proof of Lemma \ref{lem:monic factorization}]

Part 1. Observe that:
$$ \prod \limits _{i=1}^{m'} h_i^{e_i}(\xb,x_{n+1}) \cdot \prod \limits _{l=m'+1}^{m} h_l^{e_l}(\xb) \cdot f_{k}^{k-1}(\xb) 
= f(\xb,x_{n+1}) \cdot f_{k}^{k-1}(\xb) =  
\hat{f} \left( f_k \cdot x_{n+1},\xb \right) = \prod \limits _{j=1}^{\hat{m}} \hat{h}_j^{\hat{e}_j} \left(f_k \cdot x_{n+1},\xb \right)$$
Let us view the above as univariate polynomials in $x_{n+1}$ over  $\left( \F(x_1, \ldots, x_n) \right)[x_{n+1}]$.
Given this view, the term $\prod \limits _{l=m'+1}^m h_l^{e_l}(\xb) \cdot f_{k}^{k-1}(\xb)$ is a field element in the field of rational functions: $ \F(x_1, \ldots, x_n) $. Therefore, by Uniqueness of Factorization (Lemma \ref{lem:ufd}) $\hat{m} = m'$ and there exist polynomials $u_1(\xb), \ldots, u_{m'}(\xb) \in \F(x_1, \ldots, x_n)$ and a permutation $\sigma:[m'] \to [m']$ such that: $\hat{h}_i \left(f_k \cdot x_{n+1},\xb \right) = h_{\sigma(i)} \left( \xb,x_{n+1} \right) \cdot u_i(\xb)$  and $\hat{e}_i = e_{\sigma(i)}$ for $i \in [m']$. Finally, as $\hat{h}_i \left(f_k \cdot x_{n+1},\xb \right)$-s are  polynomials (and not rational functions) and  $h_{\sigma(i)} \left( \xb,x_{n+1} \right) \cdot u_i(\xb)$-s are irreducible polynomials, it follows that $u_1(\xb), \ldots, u_{m'}(\xb) \in \Fn$. 
\bigskip

Part 2. Observe that:
$$\prod \limits _{j=1}^r w_j^{\beta_j}(\xb) = f_k(\xb) = 
\lc_{x_{n+1}}(f) = \lc_{x_{n+1}} \left( \prod \limits _{i=1}^{m'} h_i^{e_i}(\xb,x_{n+1}) \right) \cdot \prod \limits _{l=m'+1}^{m} h_l^{e_l}(\xb).$$
and the claim follows by Uniqueness of Factorization (Lemma \ref{lem:ufd}).
\end{proof}

\newpage

\section{Open Questions}
\label{sec:open}

We conclude by listing some open problems. 

Perhaps the most immediate and natural question left open by this work is to understand whether one can obtain an improved sparsity bound on the factors of $s$-sparse polynomials of bounded individual degree. As we discussed in Section~\ref{sec:tightness}, the best lower bound for we know for the sparsity of factors of $s$-sparse polynomials of individual degree $d$ is $s^{\log d}$ over fields of characteristic $0$ and about $s^{d}$ over general fields. Thus there is a considerable gap between these lower bounds and the upper bound that we prove, and it is a very interesting question to close to gap. 

Another more ambitious goal is to obtain a non trivial sparsity bound with no restriction on individual degree. As we noted in Section~\ref{sec:tightness}, such a result would not be possible for all fields, and any such proof would have to use the properties of the underlying field to obtain a better bound. 

One could also study the algorithmic implications of a general sparsity bound. It seems challenging to derandomize polynomial factoring, even if we assume that factors of a given sparse polynomial are sparse (without assuming any individual degree bound). We leave this as an interesting open problem.

Given the result of~\cite{KSS14} which shows an equivalence between the problems of polynomial identity testing and polynomial factorization, this also naturally raises the question (and indeed it was raised in~\cite{KSS14}) of whether one can derandomize factoring for the classes of polynomials for which we know how to derandomize PIT. Sparse polynomials are a natural example of such a class, but there are several other natural classes that one could consider. 

\section*{Acknowledgments}

The authors would like to thank Ramprasad Saptharishi \cite{Sap18} for sharing with us the example presented in Claim \ref{example:Sap} and for letting us include it in this paper.
In addition, the authors would like to thank the anonymous referees for useful comments that improved the presentation of the results.

\bibliographystyle{abbrv}     

\bibliography{C:/Work/Papers/bibliography}


\end{document}